\documentclass[11pt]{article}
\usepackage{amsfonts}
\usepackage{amssymb,amsmath}
\usepackage{graphicx}
\usepackage{cite}
\usepackage{color}
\usepackage{amsthm}
\usepackage[figuresright]{rotating}
 \usepackage{subfig}
\usepackage{float}
\numberwithin{equation}{section}

\newtheorem{theorem}{Theorem}[section]
\newtheorem{definition}{Definition}[section]
\newtheorem{lemma}{Lemma}[section]

\newtheorem{remark}{Remark}[section]
\newtheorem{Assumption}{Assumption}[section]
\newtheorem{algorithm}{Algorithm}[section]
\allowdisplaybreaks

\begin{document}
\vspace{1.3cm}
\title
{Two-step inertial Bregman proximal alternating linearized minimization algorithm for nonconvex and nonsmooth problems{\thanks{Supported by Scientific Research Project of Tianjin Municipal Education Commission (2022ZD007).}}}
\author{ {\sc Chenzheng Guo{\thanks{Email: g13526199036@163.com}},
Jing Zhao{\thanks{Corresponding author. Email: zhaojing200103@163.com}}}\\
\small College of Science, Civil Aviation University of China, Tianjin 300300, China\\
}
\date{}
\date{}
\maketitle{} {\bf Abstract.}
In this paper, we study an algorithm for solving a class of nonconvex and nonsmooth nonseparable optimization problems. Based on proximal alternating linearized minimization (PALM), we propose a new iterative algorithm which combines two-step inertial extrapolation and Bregman distance. By constructing appropriate benefit function, with the help of Kurdyka--{\L}ojasiewicz property we establish the convergence of the whole sequence generated by proposed algorithm. We apply the algorithm to sparse nonnegative matrix factorization, signal recovery, quadratic fractional programming problem and show the effectiveness of proposed algorithm.
\vskip 0.4 true cm
\noindent {\bf Key words}:  Nonconvex and nonsmooth nonseparable optimization; Proximal alternating linearized minimization; Inertial extrapolation; Bregman distance; Kurdyka--{\L}ojasiewicz property.
\pagestyle{myheadings}
\section{Introduction}

\hspace*{\parindent} Since nonconvex functions usually approximate the original problem better than convex functions, a large number of problems require to solve nonconvex minimization problems. In recent years, many scholars have paid attention to study nonconvex optimization problems, and some effective and stable algorithms have been proposed. In this paper, we will consider solving the following nonconvex and nonsmooth nonseparable optimization problem:
\begin{equation}
\label{MP}
\min_{x\in \mathbb R^{l} ,y\in\mathbb R^{m}}   L(x,y)=f(x)+Q(x,y)+g(y),
\end{equation}
where $f:\mathbb{R}^l\rightarrow {(-\infty,+\infty]}$, $g:\mathbb{R}^m\rightarrow {(-\infty,+\infty]}$ are  proper lower semicontinuous, $Q(x,y):\mathbb{R}^l\times \mathbb{R}^m \rightarrow \mathbb{R}$ is continuously differentiable and $\nabla Q$ is  Lipschitz continuous  on bounded subsets. {Assume} that $\inf_{\mathbb{R}^l\times \mathbb{R}^m}L>-\infty$, $\inf_{\mathbb{R}^l}f>-\infty$, $\inf_{\mathbb{R}^m}g>-\infty$. Note that here and throughout the paper, no convexity is assumed on the objective function. 


Many nonconvex optimization problems have coupling terms in the objective function, so many application problems can be modeled as (\ref{MP}), e.g., signal recovery and image processing \cite{NNZC,GZZF,B}, matrix decomposition \cite{BST}, quadratic fractional programming \cite{BCV}, compressed sensing \cite{ABS,DC}, applied statistics and machine learning \cite{BPC}, etc. 

Utilizing the two-block structure, a natural method to solve problem (\ref{MP}) is the alternating minimization method, which, from a given initial point $\left ( x_{0},y_{0}   \right ) \in \mathbb{R}^l  \times \mathbb{R}^m $, generates the iterative sequence $\left \{ \left ( x_{k},y_{k}   \right )  \right \} $ via the scheme:
\begin{equation}
\label{TB}
\begin{cases}
x_{k+1}\in \arg\min_{ x\in \mathbb{R}^l}\{L(x,y_k)\},\\
y_{k+1}\in \arg\min_{y\in \mathbb{R}^m}\{L(x_{k+1},y)\}.\\
\end{cases}
\end{equation}
In the literature, the alternating minimization method are also called the Gauss–Seidel method or the block coordinate descent method.

The convergence of the alternating minimization method was first studied for the convex case. If the function $L$ is convex and continuously differentiable, and it is strict convex of one argument while the other is fixed, then every limit point of the sequence $\left \{ \left ( x_{k},y_{k}   \right )  \right \} $ generated by (\ref{TB}) minimizes $L$ \cite{BT,BN,BTL}. Removing the strict convexity assumption, one can modify the alternating minimization algorithm by adding a proximal term \cite{A}, resulting in the scheme:

\begin{equation}
\label{PAMA}
\begin{cases}
x_{k+1}\in \arg\min_{ x\in \mathbb{R}^l}\{L(x,y_k)+\frac{1}{2\lambda_k}\|x-x_k\|^2_2\},\\
y_{k+1}\in \arg\min_{y\in \mathbb{R}^m}\{L(x_{k+1},y)+\frac{1}{2\mu_k}\|y-y_k\|^2_2\},\\
\end{cases}
\end{equation}
where $\{\lambda_k\}_{k\in\mathbb{N}}$ and $\{\mu_k\}_{k\in\mathbb{N}}$ are positive sequences. In \cite{ABR}, {Attouch} et al. applied (\ref{PAMA}) to solve nonconvex problem (\ref{MP}) and proved the sequence generated by the proximal alternating minimization (PAM) algorithm (\ref{PAMA}) converges to a critical point. Because the proximal alternating minimization algorithm requires an exact solution at each iteration step, the subproblems in (\ref{PAMA}) may be very hard to solve. The linearization technique is one of the effective methods to overcome the absence of an analytic solution to the subproblem. To overcome these drawbacks in the above algorithm, by the proximal linearization of each subproblem in (\ref{PAMA}), {Bolte} et al. \cite{BST} were inspired by the proximal forward-backward algorithm and proposed the proximal alternating linearized minimization (PALM) algorithm under the condition that the coupling term $Q(x,y)$ is continuously differentiable. That is to say, for the $x$-subproblem, the function $Q(x,y)$ is linearized at the point $x_k$, and for the $y$-subproblem, the function $Q(x,y)$ is linearized at the point $y_k$ that yields the following algorithm:
\begin{equation}
\label{PALM}
\begin{cases}
x_{k+1}\in \arg\min_{ x\in \mathbb{R}^l}\{f(x)+\langle x,\nabla_xQ(x_k,y_k)\rangle+\frac{1}{2\lambda_k}\|x-x_k\|^2_2\},\\
y_{k+1}\in \arg\min_{y\in \mathbb{R}^m}\{g(y)+\langle y,\nabla_yQ(x_{k+1},y_k)\rangle+\frac{1}{2\mu_k}\|y-y_k\|^2_2\}.\\
\end{cases}
\end{equation}
In this way, if proximal operator of $f$ or $g$ have a closed-form or can be easily calculated, then the problem (\ref{PALM}) can be easily solved. They proved that each bounded sequence generated by PALM globally converges to a critical point. When $f$ and $g$ are continuously differentiable, a natural idea is to linearize $f$ and $g$. Corresponding algorithm was proposed by {Nikolova} et al. \cite{NT}, called alternating structure-adapted proximal gradient descent (ASAP) algorithm. With the help of Kurdyka--{\L}ojasiewicz property they establish the convergence of the whole sequence generated by proposed algorithm.

The inertial extrapolation technique has been widely used to accelerate the iterative algorithms for convex  and nonconvex optimizations, since the cost of each iteration stays basically unchanged \cite{OCB,BC}. The inertial scheme, starting from the so-called heavy ball method of {Polyak} \cite{PB}, was recently proved to be very efficient in accelerating numerical methods, especially the first-order methods. The main  feature of the idea is that the new iteration uses the previous two or more iterations. 

Recently, there are increasing interests in studying inertial type algorithms, such as {Nicolas} et al. \cite{LNP} proposed inertial versions of block coordinate descent methods for solving nonconvex nonsmooth optimization problems. {Feng} et al. \cite{FZZZ} focused on a minimization optimization model that is nonconvex and nonsmooth and established an inertial Douglas–Rachford splitting (IDRS) algorithm, which incorporate the inertial technique into the framework of the Douglas–Rachford splitting algorithm.
Specially, for solving problem (\ref{MP}), {Zhang and He} \cite{ZH} introduced an  inertial version of the proximal alternating minimization method, {Pock and Sabach \cite{PS}} proposed the following inertial proximal alternating linearized minimization (iPALM) algorithm: 
\begin{equation}
\label{iPALM}
\begin{cases}
u_{1k}=x_k+\alpha _{1k}(x_k-x_{k-1}), v_{1k}=x_k+\beta _{1k}(x_k-x_{k-1}),\\
x_{k+1}\in \arg\min_{ x\in \mathbb{R}^l}\{f(x)+\langle x,\nabla_xQ(v_{1k},y_k)\rangle+\frac{1}{2\lambda_k}\|x-u_{1k}\|^2_2\},\\
u_{2k}=y_k+\alpha _{2k}(y_k-y_{k-1}), v_{2k}=y_k+\beta _{2k}(y_k-y_{k-1}),\\
y_{k+1}\in \arg\min_{y\in \mathbb{R}^m}\{g(y)+\langle y,\nabla_yQ(x_{k+1},v_{2k})\rangle+\frac{1}{2\mu_k}\|y-u_{2k}\|^2_2\},\\
\end{cases}
\end{equation}
where $\alpha _{1k},\alpha _{2k},\beta _{1k},\beta _{2k}\in \left [ 0,1 \right ] $. They proved that the generated sequence globally converges to critical point of the objective function under the condition of the Kurdyka--{\L}ojasiewicz property. When $\alpha _{1k}=\alpha _{2k}=\beta _{1k}=\beta _{2k}=0$, iPALM reduces to 
PALM. Then {Cai} et al. \cite{GCH} presented a Gauss--Seidel type inertial proximal alternating linearized minimization (GiPALM) algorithm for solving problem \eqref{MP}:
\begin{equation}
\label{GiPALM}
\begin{cases}
x_{k+1}\in \arg\min_{ x\in \mathbb{R}^n}\{f(x)+\langle x,\nabla_xQ(\tilde{x}_{k},\tilde{y}_k)\rangle+\frac{1}{2\lambda_k}\|x-\tilde{x}_{k}\|^2_2\},\\
\tilde{x}_{k+1}=x_{k+1}+\alpha (x_{k+1}-\tilde{x}_{k}), \alpha \in [0,1),\\
y_{k+1}\in \arg\min_{y\in \mathbb{R}^m}\{g(y)+\langle y,\nabla_yQ(\tilde{x}_{k+1},\tilde{y}_{k})\rangle+\frac{1}{2\mu_k}\|y-\tilde{y}_{k}\|^2_2\},\\
\tilde{y}_{k+1}=y_{k+1}+\beta(y_{k+1}-\tilde{y}_{k}), \beta \in [0,1).
\end{cases}
\end{equation}
In order to use the existing information as much as possible to further improve the numerical performance, {Han} et al. \cite{XX} proposed a new inertial version of proximal alternating linearized minimization (NiPALM) algorithm, which inherits both advantages of iPALM and GiPALM. And {Xu} et al. \cite{YX} proposed accelerated alternating structure-adapted proximal gradient descent algorithm by using inertial extrapolation technique.

Bregman distance is a useful substitute for a distance, obtained from the various choices of functions. The applications of Bregman distance instead of the norm give us alternative ways for more flexibility in the selection of regularization. Choosing appropriate Bregeman distances can obtain colsed form of solution for solving subproblem. Bregman distance regularization is also an effective way to improve the numerical results of the algorithm. 
In {\cite{ZQ}}, the authors constructed the following two-step inertial Bregman alternating minimization (TiBAM) algorithm using the information of the previous three iterates: 
\begin{equation}
\label{TiBAM}
\begin{cases}
x_{k+1}\in \arg\min_{ x\in \mathbb{R}^l}\{L(x,y_k)+D_{\phi_1}(x,x_k)+\alpha_{1k} \langle x,x_{k-1}-x_k\rangle+\alpha_{2k} \langle x,x_{k-2}-x_{k-1}\rangle\},\\
y_{k+1}\in \arg\min_{y\in \mathbb{R}^m}\{L(x_{k+1},y)+D_{\phi_2}(y,y_k)+\beta_{1k} \langle y,y_{k-1}-y_k\rangle+\beta_{2k} \langle y,y_{k-2}-y_{k-1}\},
\end{cases}
\end{equation}
where $D_{\phi_i}(i=1,2)$ denote the Bregman distance with respect to  $\phi_i(i=1,2)$, respectively. The convergence is obtained provided an appropriate regularization of the objective function satisfies the Kurdyka--{\L}ojasiewicz inequality. Based on alternating minimization algorithm, {Chao} et al. \cite{Z} proposed  the following inertial alternating minimization with Bregman distance (BIAM) algorithm:
\begin{equation}
\label{BIAM}
\begin{cases}
x_{k+1}\in \arg\min_{ x\in \mathbb{R}^l}\{f(x)+Q(x,\hat{y} _k)+\lambda _{k} D_{\phi_1}(x,\hat{x} _k)\},\\
\hat{x}_{k+1}=x_{k+1}+\alpha(x_{k+1}-\hat x_{k}), \alpha \in [0,1),\\
y_{k+1}\in \arg\min_{ y\in \mathbb{R}^m}\{g(y)+Q(\hat{x}_{k+1},y)+\mu _{k} D_{\phi_2}(y,\hat{y} _k)\},\\
\hat{y}_{k+1}=y_{k+1}+\beta(y_{k+1}-\hat y_{k}), \beta \in [0,1).
\end{cases}
\end{equation}
Suppose that the benefit function satisfies the Kurdyka--{\L}ojasiewicz property and the parameters are selected appropriately, they proved the convergence of BIAM algorithm.

In this paper, based on the proximal alternating minimization algorithm, we combine inertial technique, linearization technique and Bregman distance to construct two-step inertial Bregman proximal alternating linearized minimization algorithm using the information from the previous three iterations, and give a global convergence analysis of the algorithm under the assumption that the benefit function satisfies the Kurdyka--{\L}ojasiewicz property.

The article is organized as follows: In Section 2, we recall some concepts and important lemmas which will be used in the proof of main results. In Section 3, we present the two-step inertial Bregman proximal alternating linearized minimization algorithm and show its convergence. Finally, in Section 4, the preliminary numerical examples on sparse nonnegative matrix factorization, signal recovery and quadratic fractional programming problem are provided to illustrate the behavior of the proposed algorithm.
\section{Preliminaries}
\subsection{Subdifferentials of nonconvex and nonsmooth functions}
\hspace*{\parindent}Consider the Euclidean vector space $\mathbb{R}^d$ of dimension $d\geq 1$ and denote the standard inner product by $\langle\cdot,\cdot\rangle$ and the induced norm by $\|\cdot\|_2$. We use $\omega(x_k)=\{x:\exists x_{k_j} \rightarrow x\}$ to stand for the limit set of $\{x_k\}_{k\in \mathbb{N}}$.

Given a function $f: \mathbb{R}^d \rightarrow(-\infty,+\infty]$, the {\it domain} and the {\it epigraph} of $f$ are defined by 
dom$f:=\{x\in \mathbb{R}^d: f(x)<+\infty\}$ and epi$f:=\{(x,\alpha )\in \mathbb{R}^d\times \mathbb{R}: f(x)\le \alpha \}$,
respectively. We say that $f$ is {\it proper} if dom$f\neq \emptyset$, $f$ is {\it lower semicontinuous} if epi$f$ is closed. Further we recall some generalized subdifferential notions and the basic properties which are needed in the paper, see \cite{MV,BDL}.
\begin{definition}
\rm
\label{Def21}(Subdifferentials)
Let $f: \mathbb{R}^d \rightarrow(-\infty,+\infty]$ be a proper and lower semicontinuous function. For $x\in$ {\rm dom}$f $, the Fr\'{e}chet subdifferential of $f$ at $x$, written  $\hat{\partial}f(x)$, is the set of vectors $v\in \mathbb{R}^d$ which satisfy
$$\liminf_{y\rightarrow x}\frac{1}{\|x-y\|_2}[f(y)-f(x)-\langle v,y-x\rangle]\geq 0.$$
If $x\not\in$ {\rm dom}$f$, let $\hat{\partial}f(x)=\emptyset$.  The limiting-subdifferential \cite{MV}, or simply the subdifferential for short, of $f$ at $x\in$ {\rm dom}$f$, written $\partial f(x)$, is defined as follows:
$$\partial f(x):=\{v\in \mathbb{R}^d: \exists x_k\rightarrow x, f(x_k)\rightarrow f(x), v_k\in \hat{\partial}f(x_k), v_k\rightarrow v\}.$$
\end{definition}

\begin{remark}
\label{rem21}
\rm
(a) The above definition implies that $\hat{\partial}f(x)\subseteq \partial f(x)$ for each $x\in \mathbb{R}^d$, where the first set is convex and closed while the second one is closed (see\cite{RW}).

(b) (Closedness of $\partial f$)  Let $\{x_k\}_{k\in \mathbb{N}}$ and $\{v_k\}_{k\in \mathbb{N}}$ be sequences in $\mathbb{R}^d$ such that $v_k \in \partial f(x_k)$ for all $k\in \mathbb{N}$. If $(x_k,v_k)\rightarrow (x,v)$ and $f(x_k)\rightarrow f(x)$ as $k\rightarrow \infty$, then $v \in \partial f (x)$.

(c) A necessary (but not sufficient) condition for $x\in \mathbb{R}^d$ to be a minimizer of $f$ is
$$0\in \partial f(x).\eqno{(2.1)}$$
A point {satisfying} (2.1) is called limiting-critical or simply critical. The set of critical points of $f$ is denoted by crit$f$.

(d) If $f: \mathbb{R}^d \rightarrow(-\infty,+\infty]$ be a proper and lower semicontinuous and $h : \mathbb{R}^d \rightarrow \mathbb{R}$ is a continuously differentiable function, then $\partial(f+h)(x) = \partial f (x)+\nabla h(x)$ for all $x \in \mathbb{R}^d$.
\end{remark}

In what follows, we will consider the problem of finding a critical point $(x^\ast,y^\ast)\in$dom$L$. 
\subsection{The Kurdyka--{\L}ojasiewicz property}
\hspace*{\parindent}Now we state the definition of the Kurdyka--{\L}ojasiewicz property. This class of functions will play a crucial role in the convergence results of the proposed algorithm.

Let $f: \mathbb{R}^d \rightarrow(-\infty,+\infty]$ be a proper and lower semicontinuous function. For $\eta_1$, $\eta_2$ such that $-\infty <\eta_1<\eta_2 \leq +\infty$, we set
$$[\eta_1<f<\eta_2]=\{x\in \mathbb{R}^d:\eta_1<f(x)<\eta_2\}.$$
\begin{definition}
\rm
\label{Def22}(Kurdyka--{\L}ojasiewicz property)
The function  $f: \mathbb{R}^d \rightarrow(-\infty,+\infty]$  is said to have the Kurdyka--{\L}ojasiewicz (KL) property at $x^\ast\in$dom$f$ if there exist $\eta\in (0,+\infty]$, a neighborhood $U$ of $x^\ast$ and a continuous concave function $\varphi:[0,\eta)\rightarrow \mathbb{R}_{+}$ such that $\varphi(0)=0$, $\varphi$ is $C^1$ on $(0,\eta)$, for all $s\in(0,\eta)$ it is $\varphi'(s)>0$ and for all $x$ in $U\cap[f(x^\ast)<f<f(x^\ast)+\eta]$ the Kurdyka--{\L}ojasiewicz  inequality holds, $$\varphi'(f(x)-f(x^\ast)){\rm dist}(0,\partial f(x))\geq 1.$$
Proper lower semicontinuous functions which satisfy the Kurdyka--{\L}ojasiewicz inequality at each point of {its domain} are called KL functions.
\end{definition}

A remarkable aspect of KL functions is that they are ubiquitous in applications, typical functions that have the KL property include strongly convex functions, real analytic functions, semi-algebraic functions \cite{ABR}, subanalytic functions \cite{WCX} and log-exp functions. We note that the $L_q$ quasi-norm $\|x\|q:= (\Sigma_i|x_i|^q)^{1/q}$ with $0 < q< 1$, the sup-norm $\|x\|_\infty:= \max_i|x_i|$, $\|Ax-b\|_q^q$ with $0 < q< 1$ and $\|Ax-b\|_\infty$ are all semi-algebraic functions for any matrix $A$, and so they have the KL property. 

It is known that both real analytic and semi-algebraic functions are subanalytic. Generally speaking, the sum of two subanalytic functions is not necessarily subanalytic. However, for two subanalytic functions, if at least one function maps
bounded sets to bounded sets, then their sum is also subanalytic, as shown in \cite{XY}. In particular, the sum of a subanalytic function and an analytic function is subanalytic.  Typical subanalytic functions
include: $\|Ax-b\|^2_2+\lambda\|y\|_q^q$, $\|Ax-b\|_2^2+\lambda\Sigma_i(y_i^2 +\varepsilon)^{q/2}$ and so on.
\subsection{Bregman distance}
\begin{definition}
\rm
{ A function $f$ is said convex if dom$f$ is a convex set and if, for all $x$, $y\in$dom$f$, $\alpha\in[0,1]$,
$$f(\alpha x+(1-\alpha)y)\leq \alpha f(x)+(1-\alpha)f(y).$$
$f$ is said $\theta$-strongly convex with $\theta> 0$ if $f-\frac{\theta}{2}\|\cdot\|^2$ is convex, i.e.,
$$f(\alpha x+(1-\alpha)y)\leq \alpha f(x)+(1-\alpha)f(y)-\frac{1}{2}\theta\alpha(1-\alpha)\|x-y\|^2$$
for all $x$, $y\in$dom$f$ and  $\alpha\in[0,1]$.}
\end{definition}
Suppose that the function $f$ is differentiable. Then $f$ is convex if and only if dom$f$ is a convex set and
$$f(x)\ge f(y)+\langle \nabla f(y),x-y\rangle$$
holds for all $x$, $y\in$dom$f$. Moreover, $f$ is $\theta$-strongly convex with $\theta> 0$ if and only if
$$f(x)\ge f(y)+\langle \nabla f(y),x-y\rangle+\frac{\theta}{2}\|x-y\|^2$$
for all $x$, $y\in$dom$f$.

The Bregman distance is briefly described below, as a generalization of {the} squared Euclidean distance, the Bregman distance shares many similar nice properties of the Euclidean distance. However, the Bregman distance is not a real metric, since it does not satisfy the triangle inequality nor symmetry. But the use of the Bregman distance makes the algorithm more widely applicable, so the Bregman distance plays an important role in various iterative algorithms.

\begin{definition}
\rm
\label{Defbreg}
Let $\phi:\mathbb{R}^d \rightarrow(-\infty,+\infty]$ be a convex and G\^{a}teaux differentiable function.
The function $D_\phi :$ dom$\phi\,\,\times$ intdom$\phi \rightarrow [0,+\infty)$, defined by
$$D_\phi(x,y)=\phi(x)-\phi(y)-\langle \nabla\phi(y),x-y\rangle,$$
is called the Bregman distance with respect to $\phi$.
\end{definition}

{From the above definition,  it follows  that
$$D_\phi(x,y)\geq\frac{\theta}{2}\|x-y\|^2,
\eqno{(2.2)}$$
if $\phi$ is $\theta$-strongly convex. Furthermore, it is easy to verify that $D_\phi(x,y)\geq 0$ and $D_\phi(x,y)\geq D_\phi(z,y)$ if $z\in [y,x]$.  We take the non-Euclidean distance \cite{HKM} as an example. For a positive definite
matrix $M$,  let $\|x\|_M:=\sqrt{\langle x,Mx\rangle}$ and take $\phi(x)=\|x\|_M^2$. Then $D_\phi(x,y)=\|x-y\|_M^2$. In particular, if $M$ is an identity, $D_\phi(x,y)$ reduces to the Euclidean squared distance.
\subsection{Proximal operator for nonconvex functions}
\begin{definition}
\rm
\label{Def25}(Well-definedness of proximal operator)
Let $\sigma:\mathbb{R}^d \rightarrow(-\infty,+\infty]$ be a proper and lower semicontinuous function with $\inf_{\mathbb{R}^d}\sigma>-\infty$.
Given $x\in \mathbb{R}^d$ and $t> 0$, the proximal operator associated to $\sigma$ is defined respectively by
$$prox_{t\sigma} \left ( x \right ) :=\arg\min_{ u\in \mathbb{R}^d}\left \{ \sigma \left (u  \right ) +\frac{1}{2t}\left \| u-x \right \|^2_2  \right \}.
\eqno{(2.3)}$$\  
Then, for every $t\in \left ( 0,\infty  \right )$ the set $prox_{\frac{1}{t}\sigma } \left ( x \right )$ is nonempty and compact.
\end{definition}

{It is well-known that the proximal forward-backward scheme for minimizing the sum of a smooth function $h$ with a nonsmooth one $\sigma$ can simply be viewed as the proximal regularization of $h$ linearized at a given point $x^{k} $:
$$x_{k+1}\in \arg\min_{ x\in \mathbb{R}^l}\{\sigma (x)+\langle x-x_k,\nabla h(x_k)\rangle+\frac{t}{2} \|x-x_k\|^2_2\},\ (t>0),\eqno{(2.4)}\\$$
that is, using the proximal operator notation defined in (2.3), we get
$$x_{k+1} \in prox_{\frac{1}{t}\sigma } \left (x_{k}  -\frac{1}{t}  \nabla h(x_k)\right ).$$

\subsection{Some basic convergence properties}
\hspace*{\parindent}In the following we give some basic properties which are needed in the convergence analysis. We begin by recalling the well-known and important descent lemma, see \cite{BT,MV}.
\begin{lemma}{\rm (Descent lemma\cite{BT})}
\label{lem21}
Let $h: \mathbb{R}^{d}\rightarrow \mathbb{R}$ be a continuously differentiable function with gradient $\nabla h$ assumed $L_{h}$-Lipschitz continuous. Then
$$\left | h(u)-h(v)-\left \langle u-v,\nabla h(v) \right \rangle  \right | \le \frac{L_{h}}{2}\left \| u-v \right \| ^{2}_2 ,\ \forall u,v\in \mathbb R^{d}. \eqno {\rm(2.5)} $$
\end{lemma}

\begin{Assumption}
\label{Assumption21}
\rm
 Let  $\{u_k\}_{k\in \mathbb{N}}:=\{(z_k, z_{k-1}, z_{k-2})\}_{k\in \mathbb{N}}$ be a sequence in $\mathbb{R}^{3p}$, where $z_k, z_{k-1},\\
z_{k-2}\in \mathbb{R}^p$. For convenience we use the abbreviation $\triangle_k := \|z_k-z_{k-1}\|_2$ for $k\in \mathbb{N}$.  Let $F : \mathbb{R}^{3p}\rightarrow {(-\infty,+\infty]}$ be a proper lower semicontinuous function. Then, letting $a$ and $b$ are fixed positive constants, we consider {a} sequence $\{u_k\}_{k\in \mathbb{N}}$
{satisfying} the following conditions.

(H1) (Sufficient decrease condition). For each $k\in \mathbb{N}$, it holds that
$$F(u_{k+1})+ a\triangle_{k+1}^2\leq F(u_k),$$
i.e., $F(z_{k+1}, z_{k}, z_{k-1})+ a\|z_{k+1}-z_k\|_2^2\leq F(z_k, z_{k-1}, z_{k-2})$ for $u_k=(z_k,z_{k-1},z_{k-2})$;\\

(H2) (Relative error condition). For each $k\in \mathbb{N}$, there exists $w_{k+1}\in \partial F(u_{k+1})$ such that
$$\|w_{k+1}\|_2\leq b(\triangle_{k+1}+\triangle_k+\triangle_{k-1}),$$
i.e., ${\|w_{k+1}\|_2\leq b(\|z_{k+1}-z_k\|_2+\|z_k-z_{k-1}\|_2+ \|z_{k-1}-z_{k-2}\|_2)}$;\\

(H3) (Continuity condition). {There exists a subsequence $\{u_{k_j}\}_{j\in \mathbb{N}}$ and $\tilde{u}$} such that
$$u_{k_j}\rightarrow \tilde{u}\ \ {\rm and}\ \ F(u_{k_j})\rightarrow F(\tilde{u})$$
as $j\rightarrow\infty$.

\noindent
\end{Assumption}
An abstract convergence result for inexact descent methods (see\cite{ZQ}) is given below, which can be applied to the two-step inertial algorithm. So it will be used to analyze the convergence of the proposed algorithm.
\begin{lemma}
\label{lem22}
{\rm(convergence to a critical point\cite{ZQ})} {\it  Let $F: \mathbb{R}^{3p}\rightarrow {(-\infty,+\infty]}$ be a proper lower semicontinuous function  and $\{u_k\}_{k\in \mathbb{N}}= \{(z_k, z_{k-1},z_{k-2})\}_{k\in \mathbb{N}}$ be a sequence which satisfies conditions {\rm(H1)}, {\rm(H2)} and {\rm(H3)}. Suppose $F$ have the KL property at the cluster point $\tilde{u}$ specified in {\rm(H3)}.
Then the sequence $\{z_k\}_{k\in \mathbb{N}}$ has a finite length, i.e., $\sum_{k=0}^{\infty}\|z_{k+1}-z_k\|_2<+\infty$. Moreover the sequence $\{z_k\}_{k\in \mathbb{N}}$  converges to $\bar{z}=\tilde{z}$
as $k \rightarrow\infty$, $\bar{u}=(\bar{z},\bar{z},\bar{z})$ is a critical point of $F$ where $\tilde{u}=(\tilde{z}, \tilde{z}, \tilde{z})$.}
\end{lemma}
\section{Two-step inertial Bregman proximal alternating linearized minimization algorithm}

\hspace*{\parindent}In this section we investigate the convergence property of the two-step inertial Bregman proximal alternating linearized minimization algorithm for solving nonconvex and nonsmooth nonseparable optimization problems. We assume that $\phi_i$ is $\theta_i$-strongly convex and  $\nabla\phi_i$ is $L_{\nabla\phi_i}$-Lipschitz continuous for $i=1, 2$. Set $\theta=\min\{\theta_1,\theta_2\}$.

\begin{algorithm}
\rm
Choose  $(x_0,y_0)\in{\text dom}L$ and set  $(x_{-i},y_{-i})=(x_0,y_0)$, $i=1, 2$. Take the sequences $\{\alpha_{1k}\}$, $\{\beta_{1k}\}\subseteq[0,\alpha_1]$ and $\{\alpha_{2k}\}$, $\{\beta_{2k}\}\subseteq[0,\alpha_2]$, where $\alpha_1\geq0$ and $\alpha_2\geq0$. For $k\geq 0$, let
$$
\begin{cases}
\aligned
x_{k+1}\in \arg\min_{ x\in \mathbb{R}^l}\{&f(x)+\langle x,\nabla_xQ(x_k,y_k)\rangle+D_{\phi_1}(x,x_k)+\alpha_{1k} \langle x,x_{k-1}-x_k\rangle\\
&+\alpha_{2k} \langle x,x_{k-2}-x_{k-1}\rangle\},\\
y_{k+1}\in \arg\min_{y\in \mathbb{R}^m}\{&g(y)+\langle y,\nabla_yQ(x_{k+1},y_k)\rangle+D_{\phi_2}(y,y_k)+\beta_{1k} \langle y,y_{k-1}-y_k\rangle\\
&+\beta_{2k} \langle y,y_{k-2}-y_{k-1}\rangle\},
\endaligned
\end{cases}
\eqno {\rm(3.1)}
$$
where $D_{\phi_1}$ and $D_{\phi_2}$ denote the Bregman distance with respect to  $\phi_1$ and $\phi_2$, respectively.
\end{algorithm}

Since $f$ and $g$ are proper, lower semicontinuous, $f$  and $g$ are bounded from below, and $\phi_i$ ($i=1,2$) is strongly convex,  we get that it makes sense to define the iterative scheme, meaning that the existence of $(x_{k+1},y_{k+1})$ is guaranteed for each $k\geq 0$.
{\begin{remark}\rm We give special cases of Algorithm 3.1.
\begin{itemize}
\item[(i)]
Let $\lambda$ and $\mu$ are positive constants. If we take $\phi_1(x)=\frac{1}{2\lambda}\|x\|^2_2$ and $\phi_2(y)=\frac{1}{2\mu}\|y\|^2_2$ for all $x\in \mathbb{R}^l$ and $y\in \mathbb{R}^m$, then the iterative Algorithm 3.1 becomes
$$
\begin{cases}
\aligned
x_{k+1}\in \arg\min_{ x\in \mathbb{R}^l}\{&f(x)+\langle x,\nabla_xQ(x_k,y_k)\rangle+\frac{1}{2\lambda}\|x-x_k\|^2_2+\alpha_{1k} \langle x,x_{k-1}-x_k\rangle\\
&+\alpha_{2k} \langle x,x_{k-2}-x_{k-1}\rangle\},\\
y_{k+1}\in \arg\min_{y\in \mathbb{R}^m}\{&g(y)+\langle y,\nabla_yQ(x_{k+1},y_k)\rangle+\frac{1}{2\mu}\|y-y_k\|^2_2+\beta_{1k} \langle y,y_{k-1}-y_k\rangle\\
&+\beta_{2k} \langle y,y_{k-2}-y_{k-1}\rangle\}.
\endaligned
\end{cases}
\eqno {\rm(3.2)}
$$
In this case, letting $\alpha_{2k}\equiv\beta_{2k}\equiv 0$ for all $k\geq 0$,
  (3.2) becomes the one-step inertial algorithm
$$
\begin{cases}
\aligned
x_{k+1}\in \arg\min_{ x\in \mathbb{R}^l}\{&f(x)+\langle x,\nabla_xQ(x_k,y_k)\rangle+\frac{1}{2\lambda}\|x-x_k\|^2_2+\alpha_{1k} \langle x,x_{k-1}-x_k\rangle\},\\
y_{k+1}\in \arg\min_{y\in \mathbb{R}^m}\{&g(y)+\langle y,\nabla_yQ(x_{k+1},y_k)\rangle+\frac{1}{2\mu}\|y-y_k\|^2_2+\beta_{1k} \langle y,y_{k-1}-y_k\rangle\}.
\endaligned
\end{cases}
$$
Letting $\alpha_{2k}\equiv\beta_{2k}=\alpha_{1k}\equiv\beta_{1k}\equiv 0$ for all $k\geq 0$,  (3.2) becomes the proximal alternating linearized minimization algorithm (\ref{PALM}) with $\lambda_k\equiv\lambda$ and $\mu_k\equiv\mu$ for $k\geq 0$.
\item[(ii)]
Letting $\alpha_{2k}\equiv\beta_{2k}\equiv 0$ for all $k\geq 0$,
  (3.1) becomes
the one-step inertial Bregman proximal alternating linearized minimization algorithm (iBPALM):
$$
\begin{cases}
\aligned
x_{k+1}\in \arg\min_{ x\in \mathbb{R}^l}\{&f(x)+\langle x,\nabla_xQ(x_k,y_k)\rangle+D_{\phi_1}(x,x_k)+\alpha_{1k} \langle x,x_{k-1}-x_k\rangle\},\\
y_{k+1}\in \arg\min_{y\in \mathbb{R}^m}\{&g(y)+\langle y,\nabla_yQ(x_{k+1},y_k)\rangle+D_{\phi_2}(y,y_k)+\beta_{1k} \langle y,y_{k-1}-y_k\rangle\}.
\endaligned
\end{cases}
$$
\item[(iii)]
Letting $\alpha_{2k}\equiv\beta_{2k}=\alpha_{1k}\equiv\beta_{1k}\equiv 0$ for all $k\geq 0$,  (3.1) becomes the Bregman proximal alternating minimization algorithm (BPALM):
$$
\begin{cases}
\aligned
x_{k+1}\in \arg\min_{ x\in \mathbb{R}^l}\{&f(x)+\langle x,\nabla_xQ(x_k,y_k)\rangle+D_{\phi_1}(x,x_k)\},\\
y_{k+1}\in \arg\min_{y\in \mathbb{R}^m}\{&g(y)+\langle y,\nabla_yQ(x_{k+1},y_k)\rangle+D_{\phi_2}(y,y_k)\}.
\endaligned
\end{cases}
$$
\end{itemize}
\end{remark}}

In order to get convergence analysis of our algorithm, the following assumptions required are given.
\begin{Assumption}
\label{Assumption31}
\rm

(i) 
For any fixed $y$, the partial gradient $\nabla_{x} Q$ is globally Lipschitz with module $L_1{\left ( y \right )}$, that is,
$$\left \|\nabla_{x} Q\left ( x_{1}  ,y  \right ) - \nabla_{x} Q\left ( x_{2},y  \right ) \right \|\le L_1{\left ( y \right )}\left \| x_{1}-x_{2}   \right \|, \ \forall x_{1} ,x_{2} \in \mathbb R^{l}.  $$
Likewise, for any fixed $x$, the partial gradient $\nabla_{y} Q$ is globally Lipschitz with module $L_2{\left ( x \right )}$,
$$\left \|\nabla_{y} Q\left ( x,y_{1} \right ) - \nabla_{y} Q\left ( x,y_{2} \right ) \right \|\le L_2{\left ( x \right )}\left \| y_{1}-y_{2}   \right \| , \ \forall y_{1} ,y_{2} \in \mathbb R^{m}.  $$

(ii) For $i=1,2$, there exists $L _{i}^{+} > 0$ such that
$$\inf\left \{ L_{1}\left ( y_{k}  \right ):k\in N  \right \} \ge L _{1}^{-}\ \ {\rm and}\ \inf\left \{ L_{2}\left ( x_{k}  \right ):k\in N  \right \} \ge L _{2}^{-},\eqno{(2.5)}\\ $$
 And $\theta _{1} >L_{1}^{+}$, $\theta _{2} >L_{2}^{+}$ hold, let $\rho =\min\{\theta _{1} -L_{1}^{+}\ , \theta _{2} -L_{2}^{+}\}$.

\end{Assumption}
In all of what follows, let $\{(x_k,y_k)\}_{k\in \mathbb{N}}$ be the sequence generated by Algorithm 3.1. Furthermore, for more convenient notation we abbreviate $z_k=(x_k,y_k)$,  $\triangle_k=\|z_k-z_{k-1}\|_2$ for $k\geq 0$, and introduce an benefit function
 $$H(u,v,w):= L(u)+\frac{\alpha_1+\alpha_2}{2}\|u-v\|_2^2+\frac{\alpha_2}{2}\|v-w\|_2^2.\eqno {\rm(3.3)}$$
  Note that $H(u,v,w)= L(u)$ for $u=v=w$.

\begin{lemma}
\label{lem31}
{\it Let $2(\alpha_1+\alpha_2)<\rho $ and the sequence $\{z_k\}_{k\in \mathbb{N}}$ be generated by Algorithm 3.1. Then

{\rm (i)} the sequence $\{H(z_{k},z_{k-1},z_{k-2})\}_{k\in \mathbb{N}}$  is monotonically nonincreasing and convergent, and it holds that
$$H(z_{k+1},z_{k},z_{k-1})+a\triangle_{k+1}^2\leq H(z_{k},z_{k-1},z_{k-2}),\ k\geq 0,\eqno{(3.4)}$$
where $a=\frac{\rho-2(\alpha_1+\alpha_2)}{2}>0$;

{\rm (ii)} it holds that $\sum_{k=0}^\infty\triangle_{k}^2<+\infty$, and thus $\lim_{k\rightarrow\infty}\triangle_{k}=0$;

{\rm (iii)} the sequence $\{L(x_k,y_k)\}_{k\in \mathbb{N}}$ is convergent.
}
\end{lemma}

\begin{proof} (i) From Algorithm 3.1, we have
$$
\aligned
&f(x_{k+1} ) +\langle x _{k+1},\nabla_xQ(x_k,y_k)\rangle+D_{\phi_1}(x_{k+1},x_k)+\alpha_{1k} \langle x_{k+1},x_{k-1}-x_k\rangle\\
&+\alpha_{2k} \langle x_{k+1},x_{k-2}-x_{k-1}\rangle\\
\leq &f(x_{k} )+\langle x _{k},\nabla_xQ(x_k,y_k)\rangle+D_{\phi_1}(x_{k},x_k)+\alpha_{1k} \langle x_{k},x_{k-1}-x_k\rangle\\
&+\alpha_{2k} \langle x_{k},x_{k-2}-x_{k-1}\rangle\\
=&f(x_{k} )+\langle x _{k},\nabla_xQ(x_k,y_k)\rangle+\alpha_{1k} \langle x_{k},x_{k-1}-x_k\rangle+\alpha_{2k} \langle x_{k},x_{k-2}-x_{k-1}\rangle.
\endaligned \eqno{(3.5)}
$$
From Lemma 2.1, we have
$$Q\left ( x_{k+1},y_{k}   \right ) \le Q\left (x_{k},y_{k} \right ) +   \left \langle x_{k+1}-x_{k},\nabla_x Q\left ( x_{k},y_{k} \right )  \right \rangle+\frac{L_{1}(y_{k} ) }2{\left \| x_{k+1}-x_{k} \right \| ^{2}_2 }.\eqno{(3.6)}$$
By (3.5) and (3.6), we have
$$
\aligned
&f(x_{k+1} ) +Q(x_{k+1},y_k)+D_{\phi_1}(x_{k+1},x_k)+\alpha_{1k} \langle x_{k+1},x_{k-1}-x_k\rangle+\alpha_{2k} \langle x_{k+1},x_{k-2}-x_{k-1}\rangle\\
\le& f(x_{k+1} ) +Q\left (x_{k},y_{k} \right ) +   \left \langle x_{k+1}-x_{k},\nabla_x Q\left ( x_{k},y_{k} \right )  \right \rangle+\frac{L_{1}(y_{k} ) }2{\left \| x_{k+1}-x_{k} \right \| ^{2}_2 }+D_{\phi_1}(x_{k+1},x_k)\\
&+\alpha_{1k} \langle x_{k+1},x_{k-1}-x_k\rangle+\alpha_{2k} \langle x_{k+1},x_{k-2}-x_{k-1}\rangle\\
\le& f(x_{k} ) +Q\left (x_{k},y_{k} \right ) + \frac{L_{1}(y_{k} ) }2{\left \| x_{k+1}-x_{k} \right \| ^{2}_2 }+\alpha_{1k} \langle x_{k},x_{k-1}-x_k\rangle+\alpha_{2k} \langle x_{k},x_{k-2}-x_{k-1}\rangle,
\endaligned \eqno{(3.7)}
$$
which implies that
$$
\aligned
&L(x_{k+1},y_k)+D_{\phi_1}(x_{k+1},x_k)- \frac{L_{1}(y_{k} ) }2{\left \| x_{k+1}-x_{k} \right \| ^{2}_2 }\\
\le& L(x_{k},y_k)+\alpha_{1k} \langle x_{k+1}-x_{k},x_k-x_{k-1}\rangle+\alpha_{2k} \langle x_{k+1}-x_{k},x_{k-1}-x_{k-2}\rangle. 
\endaligned\eqno{(3.8)}$$
By (2.2) and Assumption 3.1 (ii), it follows from (3.8)  that
$$\aligned
&L(x_{k+1},y_k)+\frac{1}{2}\left ( \theta _{1} -L_{1}^{+}  \right ) \|x_{k+1}-x_k\|_2^2\\
\leq&
L(x_{k},y_k)+\alpha_{1k} \langle x_{k+1}-x_k,x_k-x_{k-1}\rangle+\alpha_{2k} \langle x_{k+1}-x_k,x_{k-1}-x_{k-2}\rangle.
\endaligned \eqno{(3.9)}$$
Similarly, we have
$$\aligned
&L(x_{k+1},y_{k+1})+\frac{1}{2}\left ( \theta _{2} -L_{2}^{+}  \right )\|y_{k+1}-y_k\|_2^2\\
\leq&
L(x_{k+1},y_k)+\beta_{1k} \langle y_{k+1}-y_k,y_k-y_{k-1}\rangle+\beta_{2k} \langle y_{k+1}-y_k,y_{k-1}-y_{k-2}\rangle.\endaligned\eqno{(3.10)}$$
Adding up the inequalities (3.9) and (3.10), we get
$$
\aligned
&L(x_{k+1},y_{k+1})+\frac{1}{2}\left ( \theta _{1} -L_{1}^{+}  \right )\|x_{k+1}-x_k\|_2^2+\frac{1}{2}\left ( \theta _{2} -L_{2}^{+}  \right )\|y_{k+1}-y_k\|_2^2\\
\leq&L(x_{k},y_k)+\alpha_{1k} \langle x_{k+1}-x_k,x_k-x_{k-1}\rangle+\alpha_{2k} \langle x_{k+1}-x_k,x_{k-1}-x_{k-2}\rangle\\
&+\beta_{1k} \langle y_{k+1}-y_k,y_k-y_{k-1}\rangle+\beta_{2k} \langle y_{k+1}-y_k,y_{k-1}-y_{k-2}\rangle\\
\leq&L(x_{k},y_k)+\frac{\alpha_{1k}}{2} (\|x_{k+1}-x_k\|_2^2+\|x_k-x_{k-1}\|_2^2)
+\frac{\alpha_{2k}}{2}(\|x_{k+1}-x_k\|_2^2+\|x_{k-1}-x_{k-2}\|_2^2)\\
&+\frac{\beta_{1k}}{2}(\|y_{k+1}-y_k\|_2^2+\|y_k-y_{k-1}\|_2^2)
+\frac{\beta_{2k}}{2}(\|y_{k+1}-y_k\|_2^2+\|y_{k-1}-y_{k-2}\|_2^2)\\
\leq& L(x_{k},y_k)+(\frac{\alpha_{1k}+\alpha_{2k}}{2} \|x_{k+1}-x_k\|_2^2+\frac{\beta_{1k}+\beta _{2k}}{2} \|y_{k+1}-y_k\|_2^2)+(\frac{\alpha_{1k}}{2} \|x_k-x_{k-1}\|_2^2\\
&+\frac{\beta_{1k}}{2} \|y_k-y_{k-1}\|_2^2)+(\frac{\alpha_{2k}}{2} \|x_{k-1}-x_{k-2}\|_2^2+\frac{\beta_{2k}}{2} \|y_{k-1}-y_{k-2}\|_2^2).
\endaligned \eqno{(3.11)}
$$
Since $\rho =\min\{\theta _{1} -L_{1}^{+}\ , \theta _{2} -L_{2}^{+}\}$, $\{\alpha_{1k}\}$, $\{\beta_{1k}\}\subseteq[0,\alpha_1]$ and $\{\alpha_{2k}\}$, $\{\beta_{2k}\}\subseteq[0,\alpha_2]$, it follows that
$$\aligned
&L(z_{k+1})+\frac{\rho }{2}\|z_{k+1}-z_k\|_2^2\\
\leq& L(z_{k})+\frac{\alpha_1+\alpha_2}{2}\|z_{k+1}-z_k\|_2^2+\frac{\alpha_1}{2}\|z_{k}-z_{k-1}\|_2^2
+\frac{\alpha_2}{2}\|z_{k-1}-z_{k-2}\|_2^2,
\endaligned$$
which implies that
$$ L(z_{k+1})+\frac{\alpha_1+\alpha_2}{2}\triangle_{k+1}^2+\frac{\alpha_2}{2}\triangle_{k}^2
+\frac{\rho -2(\alpha_1+\alpha_2)}{2}\triangle_{k+1}^2
\leq L(z_{k})+\frac{\alpha_1+\alpha_2}{2}\triangle_{k}^2+\frac{\alpha_2}{2}\triangle_{k-1}^2.$$
So, (3.4) holds and the sequence $\{H(z_k,z_{k-1},z_{k-2})\}_{k\in \mathbb{N}}$  is monotonically nonincreasing. By assumption, $L$ is bounded from below, and hence $\{H(z_k,z_{k-1},z_{k-2})\}_{k\in \mathbb{N}}$ converges.

(ii) It follows from (3.4) that
$$a\triangle_{k+1}^2\leq H(z_{k},z_{k-1},z_{k-2})-H(z_{k+1},z_{k},z_{k-1}).\eqno{(3.12)}$$
Summing up (3.12) from $k= 0,\cdots,K$ yields (note that $H(z_0, z_{-1},z_{-2})=L(z_0) $)
$$\sum_{k=0}^K a \triangle_{k+1}^2\leq L(z_0)-H(z_{K+1},z_{K},z_{K-1})\leq L(z_0)-\inf_{\mathbb{R}^l\times \mathbb{R}^m}L<\infty.$$
Let $K\rightarrow +\infty$, we have that (ii) holds.

(iii) By
$$L(x_k,y_k)=L(z_k)=H(z_{k},z_{k-1},z_{k-2})-\frac{\alpha_1+
\alpha_2}{2}\triangle_{k}^2-\frac{\alpha_2}{2}\triangle_{k-1}^2,$$
it follows from (i) and (ii) that the sequence $\{L(x_k,y_k)\}_{k\in \mathbb{N}}$ converges and
$$\lim_{k\rightarrow\infty}L(x_k,y_k)=\lim_{k\rightarrow\infty}H(z_{k},z_{k-1},z_{k-2}).$$
\end{proof}

\begin{lemma}
\label{lem32}
{\it Let the sequence $\{z_k\}_{k\in \mathbb{N}}$ be generated by Algorithm 3.1. For $k\geq 0$, define
$$
\aligned
&w_{k+1}\\
=&(\nabla_xQ(x_{k+1},y_{k+1})-\nabla_xQ(x_{k},y_{k})
+\nabla\phi_1(x_k)-\nabla\phi_1(x_{k+1}),\nabla_yQ(x_{k+1},y_{k+1})-\nabla_yQ(x_{k+1},y_{k})\\
&+\nabla\phi_2(y_k)-\nabla\phi_2(y_{k+1}),0,0,0,0)\\
&+(\alpha_{1k}(x_k-x_{k-1})+\alpha_{2k}(x_{k-1}-x_{k-2}),\beta_{1k}(y_k-y_{k-1})+\beta_{2k}(y_{k-1}-y_{k-2}),0,0,0,0)\\
&+((\alpha_{1}+\alpha_2)(x_{k+1}-x_{k}),(\alpha_{1}+\alpha_2)(y_{k+1}-y_{k}),0,0,0,0)\\
&+(0,0,(\alpha_{1}+\alpha_{2})(x_{k}-x_{k+1})+\alpha_2(x_{k}-x_{k-1}),
(\alpha_{1}+\alpha_{2})(y_{k}-y_{k+1})+\alpha_2(y_{k}-y_{k-1}),0,0)\\
&+(0,0,0,0,\alpha_{2}(x_{k-1}-x_{k}),\alpha_{2}(y_{k-1}-y_{k})).
\endaligned\eqno{(3.13)}
$$
Then$$w_{k+1}\in \partial H(z_{k+1},z_{k},z_{k-1})={\partial}H(x_{k+1},y_{k+1},x_{k},y_{k},x_{k-1},y_{k-1}).\eqno{(3.14)}$$
Moreover, if the sequence $\{z_k\}_{k\in \mathbb{N}}$ is bounded, then there exists $M>0$ such that 
$$
\aligned &\|w_{k+1}\|_2\\
\leq&(M+L_{\nabla\phi_1}+2\alpha_1+2\alpha_2)\|x_{k+1}-x_k\|_2
+(M+L_2{\left ( x _{k+1} \right )}+L_{\nabla\phi_2}+2\alpha_1+2\alpha_2)\|y_{k+1}-y_k\|_2\\
&+(\alpha_1+2\alpha_2)(\|x_{k}-x_{k-1}\|_2+\|y_{k}-y_{k-1}\|_2)+\alpha_{2}(\|x_{k-1}-x_{k-2}\|_2+\|y_{k-1}-y_{k-2}\|_2),
\endaligned \eqno{(3.15)}
$$}
\end{lemma}

\begin{proof}
 By  the  definition of $x_{k+1}$, and by Remark \ref{rem21} (c), $0$ must lie in the subdifferential at point $x_{k+1}$ of the function
 $$x\longmapsto f(x)+\langle x,\nabla_xQ(x_k,y_k)\rangle+D_{\phi_1}(x,x_k)+\alpha_{1k} \langle x,x_{k-1}-x_k\rangle+\alpha_{2k} \langle x,x_{k-2}-x_{k-1}\rangle.$$
 Since $\phi$ are differential, we have
$$0\in \partial f(x_{k+1})+\nabla_x Q(x_{k},y_k)+\nabla\phi_1(x_{k+1})- \nabla\phi_1(x_k)+\alpha_{1k}(x_{k-1}-x_k)+\alpha_{2k}(x_{k-2}-x_{k-1}),$$
which implies that
$$-\nabla_x Q(x_{k},y_k)+\nabla\phi_1(x_k)-\nabla\phi_1(x_{k+1}) +\alpha_{1k}(x_k-x_{k-1})
+\alpha_{2k}(x_{k-1}-x_{k-2})\in \partial f(x_{k+1}).\eqno{(3.16)}$$
Similarly, we have
$$0\in \partial g(y_{k+1})+\nabla_y Q(x_{k+1},y_{k})+\nabla\phi_2(y_{k+1})- \nabla\phi_2(y_k)+\beta_{1k}(y_{k-1}-y_k)+\beta_{2k}(y_{k-2}-y_{k-1}),
$$
and so
$$-\nabla_y Q(x_{k+1},y_{k})+\nabla\phi_2(y_k)-\nabla\phi_2(y_{k+1}) +\beta_{1k}(y_k-y_{k-1})+\beta_{2k}(y_{k-1}-y_{k-2})\in \partial g(y_{k+1}).\eqno{(3.17)}$$
Because of the structure of $L$,  it follows that $\partial_xL(x_{k+1},y_{k+1})=
\partial f(x_{k+1})+\nabla_x Q(x_{k+1},y_{k+1})$ and $\partial_yL(x_{k+1},y_{k+1})=
\nabla_y Q(x_{k+1},y_{k+1})+\partial g(y_{k+1})$. Hence, from (3.16) and (3.17), we have
$$\aligned 
&\nabla_x Q(x_{k+1},y_{k+1})-\nabla_x Q(x_{k},y_k)+\nabla\phi_1(x_k)-\nabla\phi_1(x_{k+1})+\alpha_{1k}(x_k-x_{k-1})+\alpha_{2k}(x_{k-1}-x_{k-2})\\
&\in\partial_xL(x_{k+1},y_{k+1}),
\endaligned\eqno{(3.18)}$$
and
$$\aligned 
&\nabla_y Q(x_{k+1},y_{k+1})-\nabla_y Q(x_{k+1},y_k)+\nabla\phi_2(y_k)-\nabla\phi_2(y_{k+1})+\beta_{1k}(y_k-y_{k-1})+\beta_{2k}(y_{k-1}-y_{k-2})\\
&\in\partial_yL(x_{k+1},y_{k+1}).
\endaligned\eqno{(3.19)}$$
Since
$$\aligned
&H(z_{k+1},z_{k},z_{k-1})\\
=&L(z_{k+1})+\frac{\alpha_1+\alpha_2}{2}\|z_{k+1}-z_{k}\|_2^2
+\frac{\alpha_2}{2}\|z_{k}-z_{k-1}\|_2^2\\
=&L(x_{k+1},y_{k+1})+\frac{\alpha_1+\alpha_2}{2}(\|x_{k+1}-x_{k}\|_2^2+\|y_{k+1}-y_{k}\|_2^2)
+\frac{\alpha_2}{2}(\|x_{k}-x_{k-1}\|_2^2+\|y_{k}-y_{k-1}\|_2^2),
\endaligned $$
we have (3.14) holds. If the sequence $\{z_k\}_{k\in \mathbb{N}}$ is bounded, noting that $\nabla Q$ is  Lipschitz continuous  on bounded subsets, then there exists $M>0$ such that 
$$\aligned
&\left \|\nabla_{x} Q\left ( x_{k+1} ,y_{k+1} \right ) - \nabla_{x} Q\left ( x_{k},y_{k} \right ) ) \right \|
\le \left \|\nabla Q\left ( x_{k+1} ,y_{k+1} \right ) - \nabla Q\left ( x_{k},y_{k} \right ) ) \right \|\\
\le &M(\left \|  x_{k+1}-x_{k} \right \|^2 +\left \|  y_{k+1}-y_{k} \right \| ^2)^{\frac{1}{2}}
\le M(\left \|  x_{k+1}-x_{k} \right \| +\left \|  y_{k+1}-y_{k} \right \| ).
\endaligned $$
Hence, the inequality (3.15) follows from the definition of the sequence $\{w_{k+1}\}_{k\in\mathbb{N}}$.
\end{proof}

\begin{lemma}
\label{lem33}
{\it Suppose that the function $L$ is coercive. Let $2(\alpha_1+\alpha_2)<\rho $ and the sequence $\{z_k\}_{k\in \mathbb{N}}$ be generated by Algorithm 3.1. { Then},

{\rm (i)} the sequence $\{z_k\}_{k\in \mathbb{N}}$ is bounded and the set $\omega(z_k)$ is nonempty;

{\rm (ii)} for any $\tilde{z}\in \omega(z_k)$  and $z_{k_j}\rightarrow \tilde{z}$ as $j\rightarrow\infty$, the sequence $H(z_{k_j},z_{k_j-1},z_{k_j-2})\rightarrow H(\tilde{z},\tilde{z},\tilde{z})$ as  $j\rightarrow\infty$ and $0\in\partial L(\tilde{z})$.}
\end{lemma}

\begin{proof} (i) By Lemma \ref{lem31} and $H(z_0,z_{-1},z_{-2})=L(z_0)$, we have
$$L(z_k)\leq H(z_k,z_{k-1},z_{k-2})\leq H(z_0,z_{-1},z_{-2})=L(z_0)$$
for $k\geq 0$. It is clear that the whole sequence $\{z_k\}_{k\in\mathbb{N}}$ is contained in the level set $\{z\in \mathbb{R}^l\times \mathbb{R}^m : \inf_{z\in \mathbb{R}^l\times \mathbb{R}^m}L(z)\leq L(z) \leq L(z_0)\}$. It follows from the coercivity of $L$ and $\inf_{z\in \mathbb{R}^l\times \mathbb{R}^m}L(z)>-\infty$ that  the sequence $\{z_k\}_{k\in\mathbb{N}}$ is bounded, and so the set $\omega(z_k)$ is nonempty.

(ii) For any $\tilde{z}\in \omega(z_k)$  and $z_{k_j}\rightarrow \tilde{z}$ as $j\rightarrow\infty$, we show that $L(z_{k_j})\rightarrow L(\tilde{z})$ as  $j\rightarrow\infty$. Let $\tilde{z}=(\tilde{x},\tilde{y})$. It following from Lemma \ref{lem31} (ii) that
$$\lim_{k\rightarrow\infty}\triangle_k=\lim_{k\rightarrow\infty}\|z_{k}-z_{k-1}\|_2=0,\eqno{(3.20)}$$
and so $\lim_{k\rightarrow\infty}\|x_{k}-x_{k-1}\|_2
=\lim_{k\rightarrow\infty}\|y_{k}-y_{k-1}\|_2=0$. From the definition of $x_{k}$, we have
$$\aligned
&f(x_{k})+\langle x _{k},\nabla_xQ(x_k,y_{k-1})\rangle+D_\phi(x_{k},x_{k-1})+\alpha_{1,k-1} \langle x_{k},x_{k-2}-x_{k-1}\rangle\\
&+\alpha_{2,k-1} \langle x_{k},x_{k-3}-x_{k-2}\rangle\\
\leq& f(\tilde{x})+\langle \tilde{x},\nabla_xQ(\tilde{x},y_{k-1})\rangle+D_\phi(\tilde{x},x_{k-1})+\alpha_{1,k-1} \langle \tilde{x},x_{k-2}-x_{k-1}\rangle\\
&+\alpha_{2,k-1} \langle \tilde{x},x_{k-3}-x_{k-2}\rangle,
\endaligned $$
which implies that
$$\aligned
&f(x_{k_j})+\langle x _{k_j},\nabla_xQ(x_{k_j},y_{k_j-1})\rangle+D_\phi(x_{k_j},x_{k_j-1})+\alpha_{1,k_j-1} \langle x_{k_j},x_{k_j-2}-x_{k_j-1}\rangle\\
&+\alpha_{2,k_j-1} \langle x_{k_j},x_{k_j-3}-x_{k_j-2}\rangle\\
\leq& f(\tilde{x})+\langle \tilde{x},\nabla_xQ(\tilde{x},y_{k_j-1})\rangle+D_\phi(\tilde{x},x_{k_j-1})+\alpha_{1,k_j-1} \langle \tilde{x},x_{k_j-2}-x_{k_j-1}\rangle\\
&+\alpha_{2,k_j-1} \langle \tilde{x},x_{k_j-3}-x_{k_j-2}\rangle.
\endaligned $$\\
Letting $j\overset{}{\rightarrow} \infty $, we have
$$\limsup_{j\rightarrow\infty}f(x_{k_j})\leq f(\tilde{x}).$$
By the lower semicontinuity of $f$, it follows that 
$$f(\tilde{x})\leq \liminf_{j\rightarrow\infty}f(x_{k_j}),$$
and so $\lim_{j\rightarrow\infty}f(x_{k_j})= f(\tilde{x})$. Similarly, we have
$\lim_{j\rightarrow\infty}g(y_{k_j})= g(\tilde{y})$, and hence $\lim_{j\rightarrow\infty}L(z_{k_j})= L(\tilde{z})$. By (3.20) and
$$H(z_{k_j},z_{k_j-1},z_{k_j-2})=L(z_{k_j})+\frac{\alpha_1+\alpha_2}{2}\|z_{k_j}-z_{k_j-1}\|_2^2
+\frac{\alpha_2}{2}\|z_{k_j-1}-z_{k_j-2}\|_2^2,$$
we have $\lim_{j\rightarrow\infty}H(z_{k_j},z_{k_j-1},z_{k_j-2})= L(\tilde{z})=H(\tilde{z},\tilde{z},\tilde{z})$.
It follows from (3.18) and (3.19) that
$$s_{k}\in \partial L(x_{k},y_{k}),$$
where
$$
\aligned
s_{k}=&(\nabla_xQ(x_{k},y_{k})-\nabla_xQ(x_{k-1},y_{k-1})+\nabla\phi_1(x_{k-1})-\nabla\phi_1(x_{k}),\nabla_y Q(x_{k},y_{k})-\nabla_y Q(x_{k},y_{k-1})\\
&+\nabla\phi_2(y_{k-1})-\nabla\phi_2(y_{k}))+(\alpha_{1,k-1}(x_{k-1}-x_{k-2})+\alpha_{2,k-1}(x_{k-2}-x_{k-3}),\\
&\beta_{1,k-1}(y_{k-1}-y_{k-2})+\beta_{2,k-1}(y_{k-2}-y_{k-3})).
\endaligned
$$
Since $\nabla Q$ is Lipschitz continuous on bounded subsets and $\nabla \phi_i$ is  Lipschitz continuous ($i=1,2$), it follows that
$s_{k}\rightarrow 0$ as $k\rightarrow\infty$. By the
closure property of the subdifferential $\partial L$, we have $0\in \partial L(\tilde{z})$.
\end{proof}

Now, using Lemma \ref{lem22}, we can verify the convergence of the sequence $\{(x_k,y_k)\}_{k\in\mathbb{N}}$ generated
by Algorithm 3.1.

\begin{theorem}
\label{th41}
{\it  Let $2(\alpha_1+\alpha_2)<\rho $ and $\{(x_k,y_k)\}_{k\in \mathbb{N}}$ be generated by Algorithm 3.1. Suppose that  the function $L$ is coercive and $z_k=(x_k,y_k)$ for all $k\in \mathbb{N}$.  Then, the sequence $\{(z_{k+1}, z_k,z_{k-1})\}_{k\in \mathbb{N}}$ satisfies {\rm(H1)}, {\rm(H2)}, and {\rm(H3)} for the function
$$H(u,v,w):\mathbb{R}^{3(l+m)}\rightarrow {(-\infty,+\infty]},$$
which is defined by (3.3). Moreover, if $H(u,v,w)$ has the KL property at a cluster point $(z^\ast,z^\ast,z^\ast)$ of the sequence $\{(z_{k+1}, z_k,z_{k-1})\}_{k\in \mathbb{N}}$. Then the sequence $\{z_k\}_{k\in \mathbb{N}}$ has  finite length, $z_k\rightarrow z^\ast$
as $k \rightarrow\infty$, and $(z^\ast,z^\ast,z^\ast)$ is a critical point of $H$. Hence, $z^\ast$ is a critical point of $L$.}
\end{theorem}

\begin{proof}
 First, we verify that assumptions (H1), (H2) and (H3) are satisfied. We consider
the sequence  $u_k=(z_k,z_{k-1},z_{k-2})$  for all $k\in \mathbb{N}$ and the proper lower semicontinuous function
$F=H$. Condition (H1) is proved in Lemma \ref{lem31} (i). To prove condition (H2), from Lemma \ref{lem32} we have
$$w_{k+1}\in \partial H(z_{k+1},z_{k},z_{k-1})=\partial H(x_{k+1},y_{k+1},x_{k},y_{k},x_{k-1},y_{k-1}),$$
where $w_{k+1}$ is defined by (3.13) for $k\in \mathbb{N}$. Lemma \ref{lem33} (i) implies that the sequence $\{z_k\}_{k\in\mathbb{N}}$ is bounded.
Let $c=\max\{L_{\nabla\phi_1},L_{\nabla\phi_2}\}, d=M+L_2(x_{k+1} )$. It follows from
$$(\|x_{k+1}-x_k\|_2+\|y_{k+1}-y_k\|_2)^2\leq 2(\|x_{k+1}-x_k\|_2^2+\|y_{k+1}-y_k\|_2^2)=2\|z_{k+1}-z_k\|_2^2$$
and (3.15) that $$
\aligned &\|w_{k+1}\|_2\\
\leq&({c+d+2\alpha_1+2\alpha_2})\sqrt{2}\|z_{k+1}-z_k\|_2+(\alpha_1+2\alpha_2)\sqrt{2}\|z_{k}-z_{k-1}\|_2
+\alpha_2\sqrt{2}\|z_{k-1}-z_{k-2}\|_2,
\endaligned
$$
which implies that condition (H2) holds for $b=\sqrt{2}(c+d+\alpha_1+2\alpha_2)$. From Lemma \ref{lem33} (i), we have $\omega(z_k)\neq\emptyset$, and so
there exists a subsequence $\{z_{k_j}\}_{j\in\mathbb{N}}\subseteq \{z_k\}_{k\in\mathbb{N}}$ such that $z_{k_j}\rightarrow \tilde{z}$ as $j\rightarrow\infty$. By Lemma \ref{lem31} (ii), we know
$\triangle_{k_j}=\|z_{k_j}-z_{k_j-1}\|_2\rightarrow 0$ as $j\rightarrow\infty$. Taking
$u_{k_j}=(z_{k_j}, z_{k_j-1},z_{k_j-2})$ and $\tilde{u}=(\tilde{z},\tilde{z},\tilde{z})$, we have $u_{k_j}\rightarrow \tilde{u}$ as $j\rightarrow\infty$. Lemma \ref{lem33} (ii) implies that
$$\lim_{j\rightarrow\infty}F(u_{k_j})=\lim_{j\rightarrow\infty}H(z_{k_j},z_{k_j-1},z_{k_j-2})
=H(\tilde{z},\tilde{z},\tilde{z})=F(\tilde{u})$$
and $0\in\partial L(\tilde{z})$. So we get that (H3) holds. Now, Lemma \ref{lem22} concludes the proof.
\end{proof}

\section{Numerical experiments}
\hspace*{\parindent}In all of the following experiments, the codes are written in MATLAB and are performed on a personal Dell computer with Pentium(R) Dual-Core CPU @ 2.4GHz and RAM 2.00GB. 
\subsection{Sparse nonnegative matrix factorization}
\ \par Given a matrix $A$, the target of nonnegative matrix factorization (NMF) \cite{DH,JN,FF} is to decompose database $A$ into a number of sparse basis, such that each database can be approximated by using a small number of those parts, i.e., seeking a factorization $A \approx XY$, where $X \in \mathbb{R} ^{ n\times r}$, $Y \in \mathbb{R} ^{r\times d}$ are nonnegative with $r \leq d$ and $X$ sparse. In order to enforce sparsity on the basis, we additionally consider $l_0$ sparsity constraint \cite{RF} in this problem. Thus sparse-NMF can be formulated as the following model
\begin{equation}
\label{(5.1)}
\aligned
\underset{X,Y}{\min}  \left \{ \left \| A-XY \right \| _{F}^{2} : \ X,Y\ge 0,\ \left \| X_i \right \| _0\le s,\ i=1,2,\dots ,r\right \} .
\endaligned
\end{equation}
Here, $X_i$ denotes the ith column of $X$. In dictionary learning and sparse coding, $X$ is called the learned dictionary with coefficients $Y$. In this formulation, the sparsity on $X$ is strictly enforced using the nonconvex $l_ 0$ constraint, restricting $75\%$ of the entries to be 0.
This model (4.1) can be converted to (1.1), where $Q(X,Y)=\frac{\lambda}{2}\left \| A-XY \right \| _{F}^{2}$, $f(X)=\sum_{i=1}^r \left \| X_i \right \| _0+\iota_{X\ge0}(X)$, $g(Y)=\iota_{Y\ge0}(Y)$, $\iota_C$ is the indicator function on $C$. 

\par We use the extended Yale-B dataset and the ORL dataset, which are standard facial recognition benchmarks consisting of human face images\footnote{ http://www.cad.zju.edu.cn/home/dengcai/Data/FaceData.html.}. The extended Yale-B dataset contains 2414 cropped images of size $32 \times 32$, while the ORL dataset contains 400 images sized $64 \times 64$, see Figure 1. For solving this sparse NMF problem (4.1), \cite{GCH,XX} gave the details on how to solve the X-subproblems and Y-subproblems. In the experiment for the Yale dataset, we extract 49 sparse basis images for the dataset. For the ORL dataset we extract 25 sparse basis images. We take parameter $\lambda=0.5$. In Algorithm 3.1, let $\phi_1(X)=\frac{\mu_1 }{2} \left \| X\right \|_{2}^{2}$, $\phi_2(Y)=\frac{\mu_2 }{2} \left \|Y \right \|_{2}^{2}$. In numerical experiment, we computed $\mu_1$ and $\mu_2$ by computing the largest eigenvalues of $\lambda YY^T$ and $\lambda X^TX$ at $k$-th iteration, respectively. We use the algorithm PALM, iPALM, GiPALM and Algorithm 3.1 for a specific sparsity setting $s=0.25$. We choose $\alpha_{1k}=\beta _{1k}=0.2$, $\alpha_{2k}=\beta _{2k}=0.3$ in Algorithm 3.1 and $\alpha_{1k}=\beta _{1k}=0.5$ in iPALM and GiPALM. The related numerical results of objective values are reported in Figure 2 and Figure 4 under Yale-B dataset and ORL dataset, respectively. And we also take extrapolation parameter $\alpha_{1k}=\alpha_{2k}=\beta _{1k}=\beta _{2k}=\frac{k-1}{k+2}$ in Algorithm 3.1 and $\alpha_{1k}=\beta _{1k}=\frac{k-1}{k+2}$ in iPALM and GiPALM. In Figure 3 and Figure 5, we report the numerical results under Yale-B dataset and ORL dataset, respectively. One can observe from these four figures that the proposed Algorithm 3.1 can get slightly lower values than the other three algorithms within almost the same computation time.
\par We also compare Algorithm 3.1 with PALM, iPALM and GiPALM for different sparsity settings (the value of $s$). We take extrapolation parameter $\alpha_{1k}=\alpha_{2k}=\beta _{1k}=\beta _{2k}=\frac{k-1}{k+2}$ in Algorithm 3.1 and $\alpha_{1k}=\beta _{1k}=\frac{k-1}{k+2}$ in iPALM and GiPALM. The results of the basis images are shown in Figure 6. One can observe from Figure 6 that for smaller values of $s$, the four algorithms lead to more compact representations. This might improve the generalization capabilities of the representation.
\begin{figure*}[!t]
    \centering
    \subfloat{\includegraphics[width=6.0in]{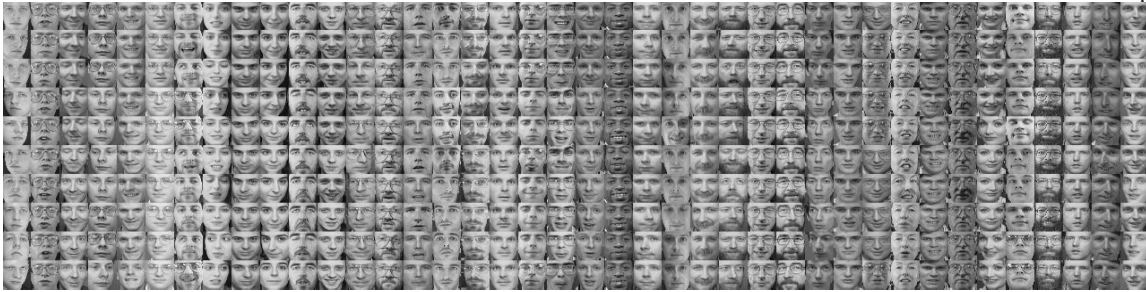}
    \label{400 normalized cropped frontal faces}}
    \caption{ORL face database which includes 400 normalized cropped frontal faces which we used in our S-NMF example.}
    \label{fig_sim}
\end{figure*}

\begin{figure*}[!t]
    \centering
    \subfloat[Epoch counts comparison]{\includegraphics[width=3.0in]{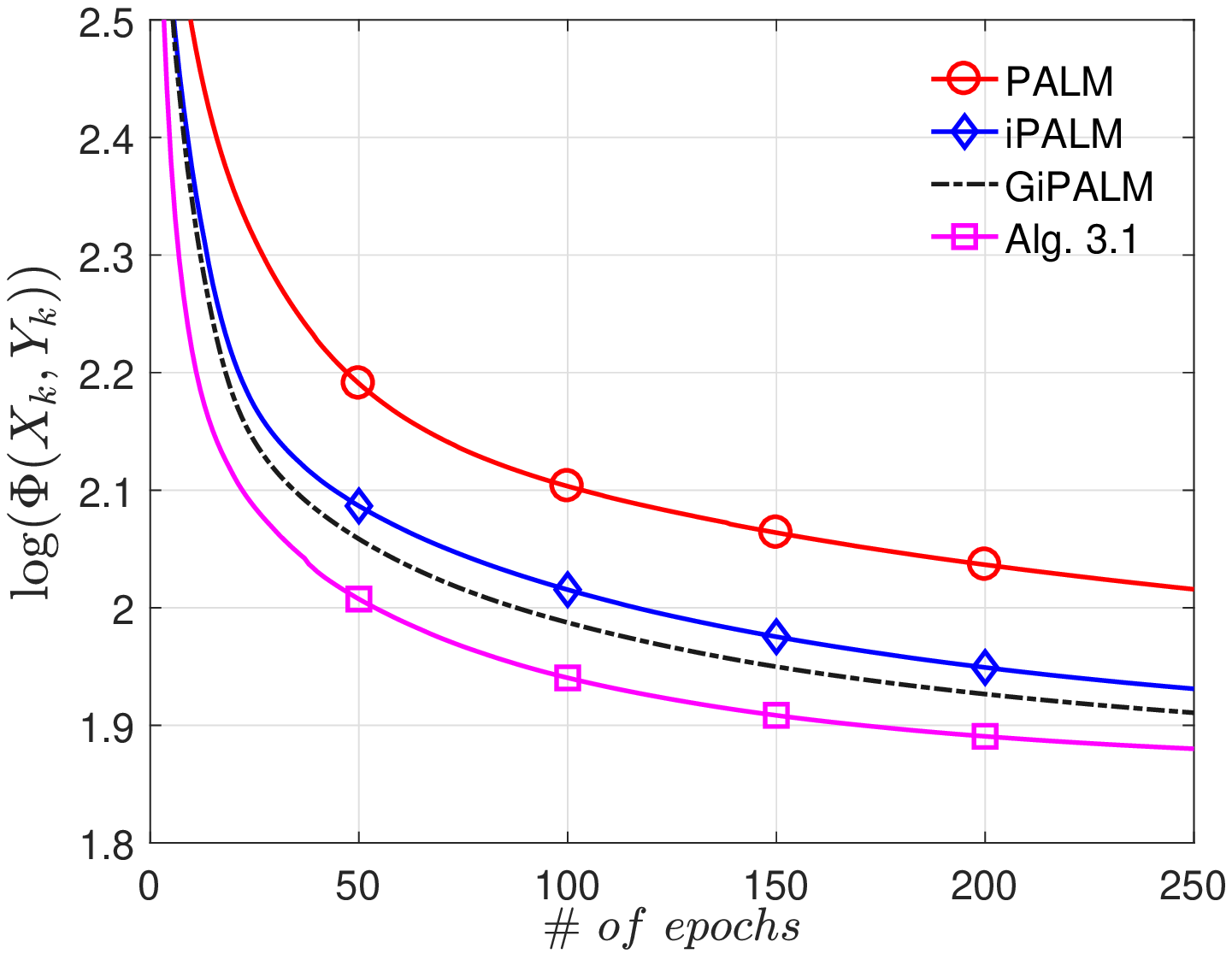}
    \label{Epoch counts comparison}}
\hfil
\subfloat[Wall-clock time comparison]{\includegraphics[width=3.0in]{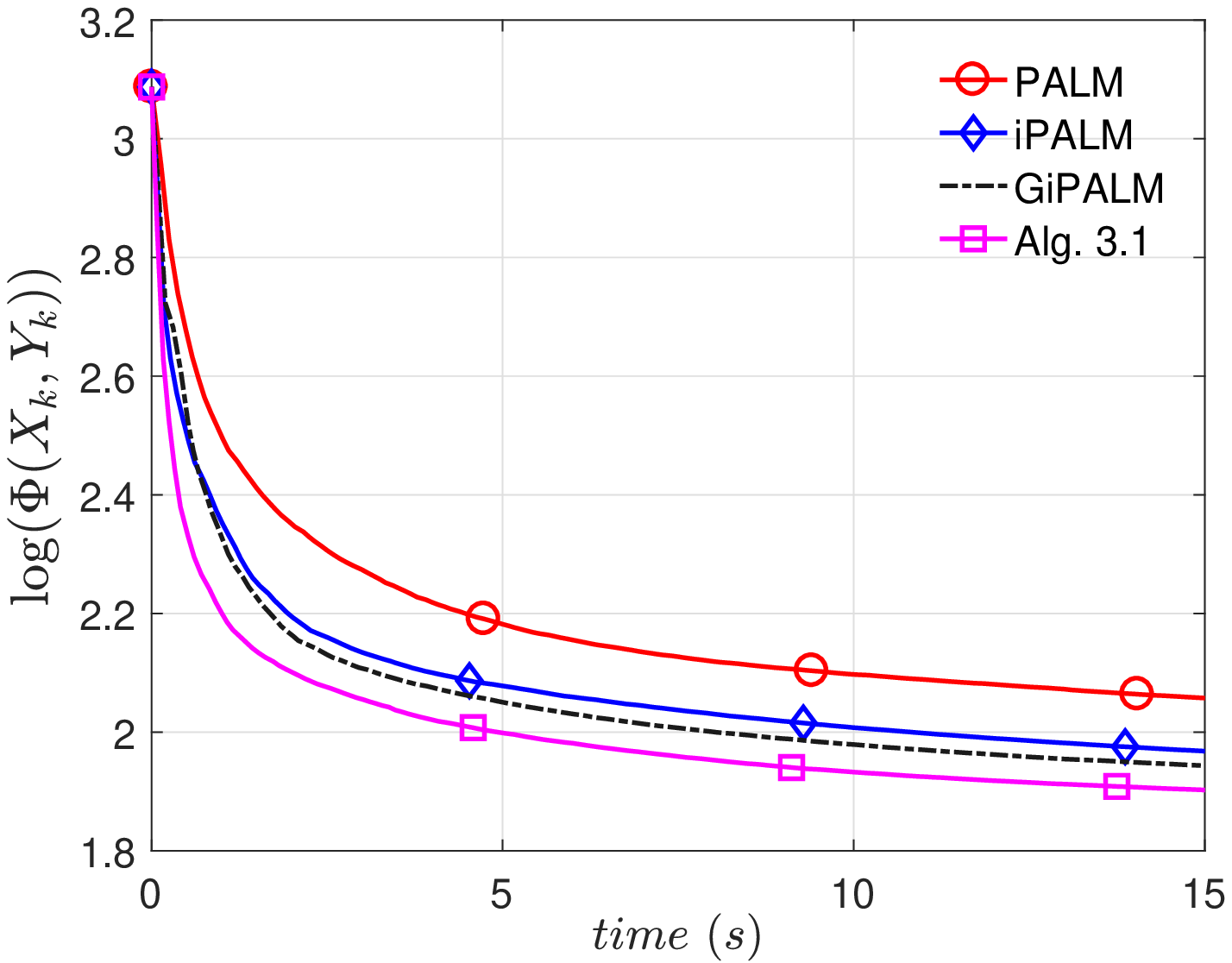}%
    \label{Wall-clock time comparison}}

    \caption{Objective decrease comparison of sparse-NMF with $s = 25\%$ on Yale dataset.}
    \label{fig_sim}
\end{figure*}

\begin{figure*}[!t]
    \centering
    \subfloat[Epoch counts comparison]{\includegraphics[width=3.0in]{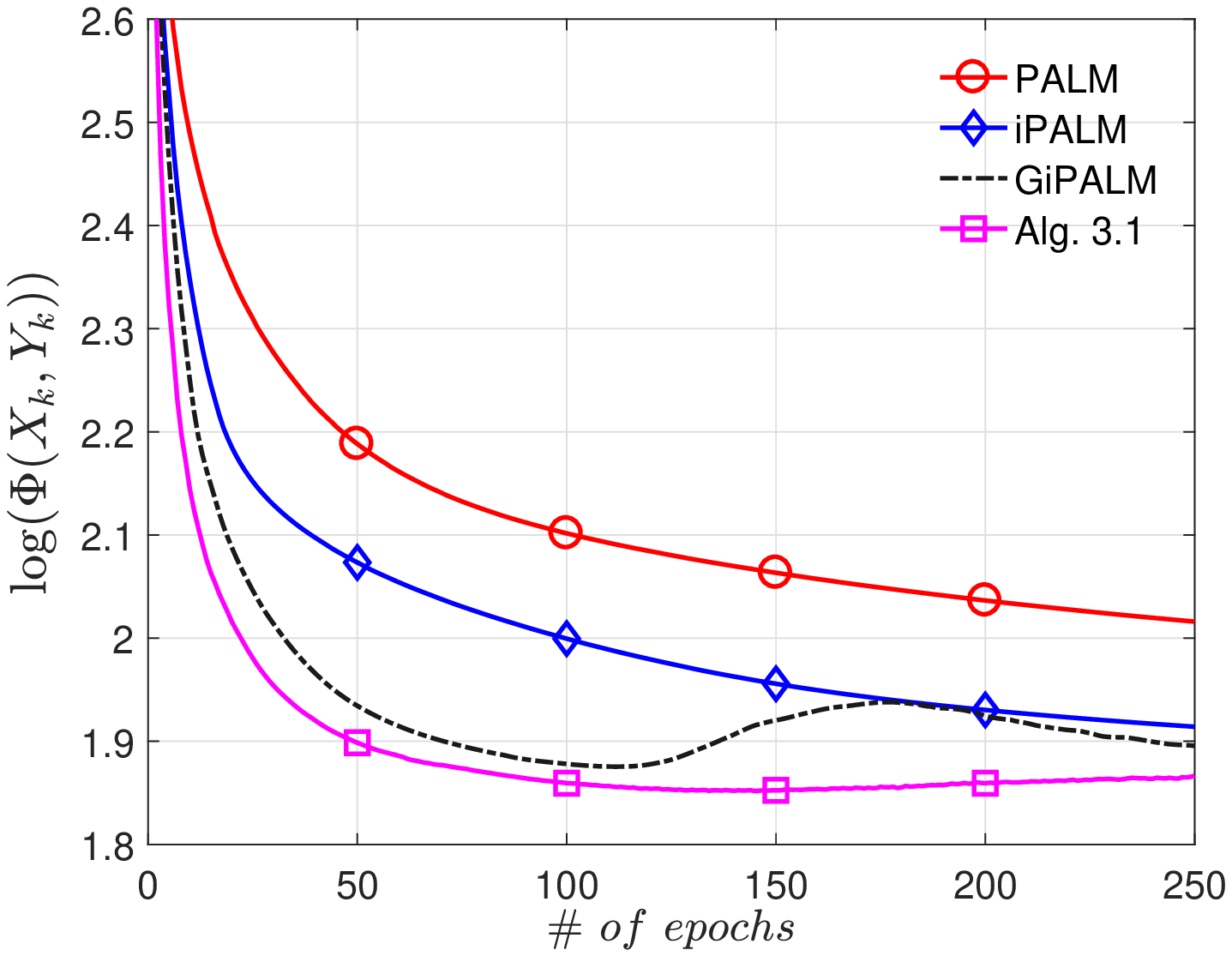}
    \label{Epoch counts comparison}}
\hfil
\subfloat[Wall-clock time comparison]{\includegraphics[width=3.0in]{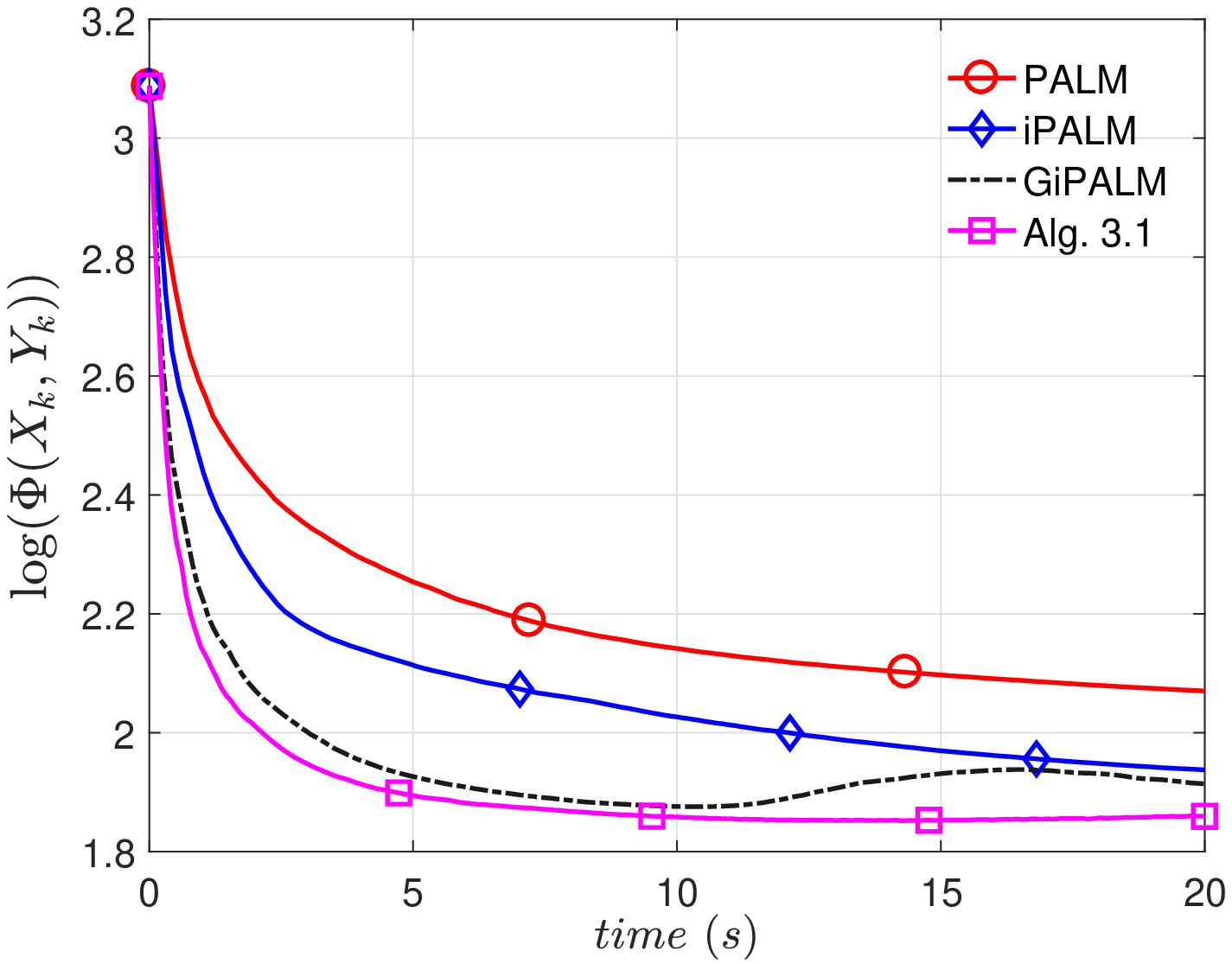}%
    \label{Wall-clock time comparison}}

    \caption{Objective decrease comparison of sparse-NMF with $s = 25\%$ on Yale dataset.}
    \label{fig_sim}
\end{figure*}

\begin{figure*}[!t]
    \centering
    \subfloat[Epoch counts comparison]{\includegraphics[width=3.0in]{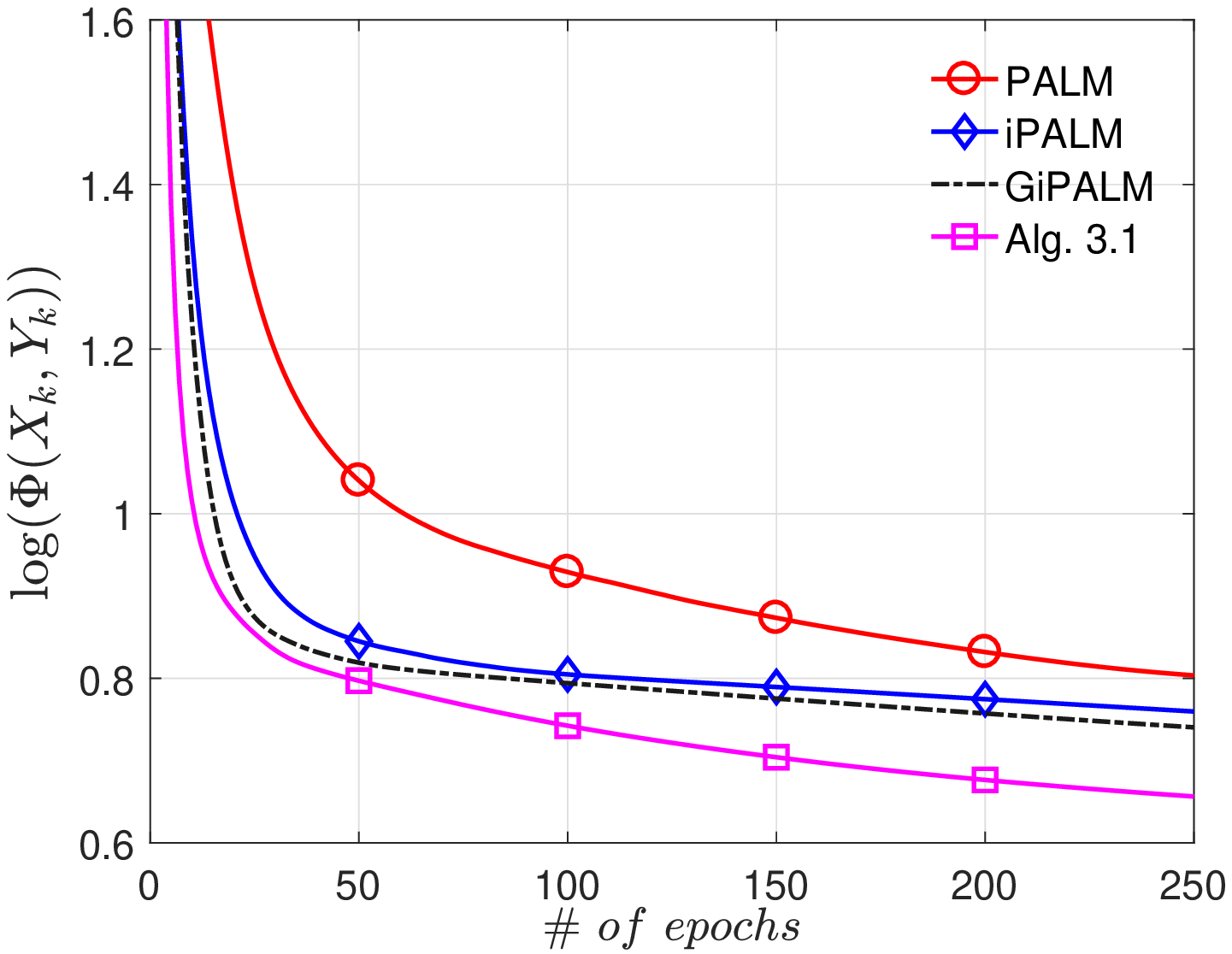}
    \label{Epoch counts comparison}}
\hfil
\subfloat[Wall-clock time comparison]{\includegraphics[width=3.0in]{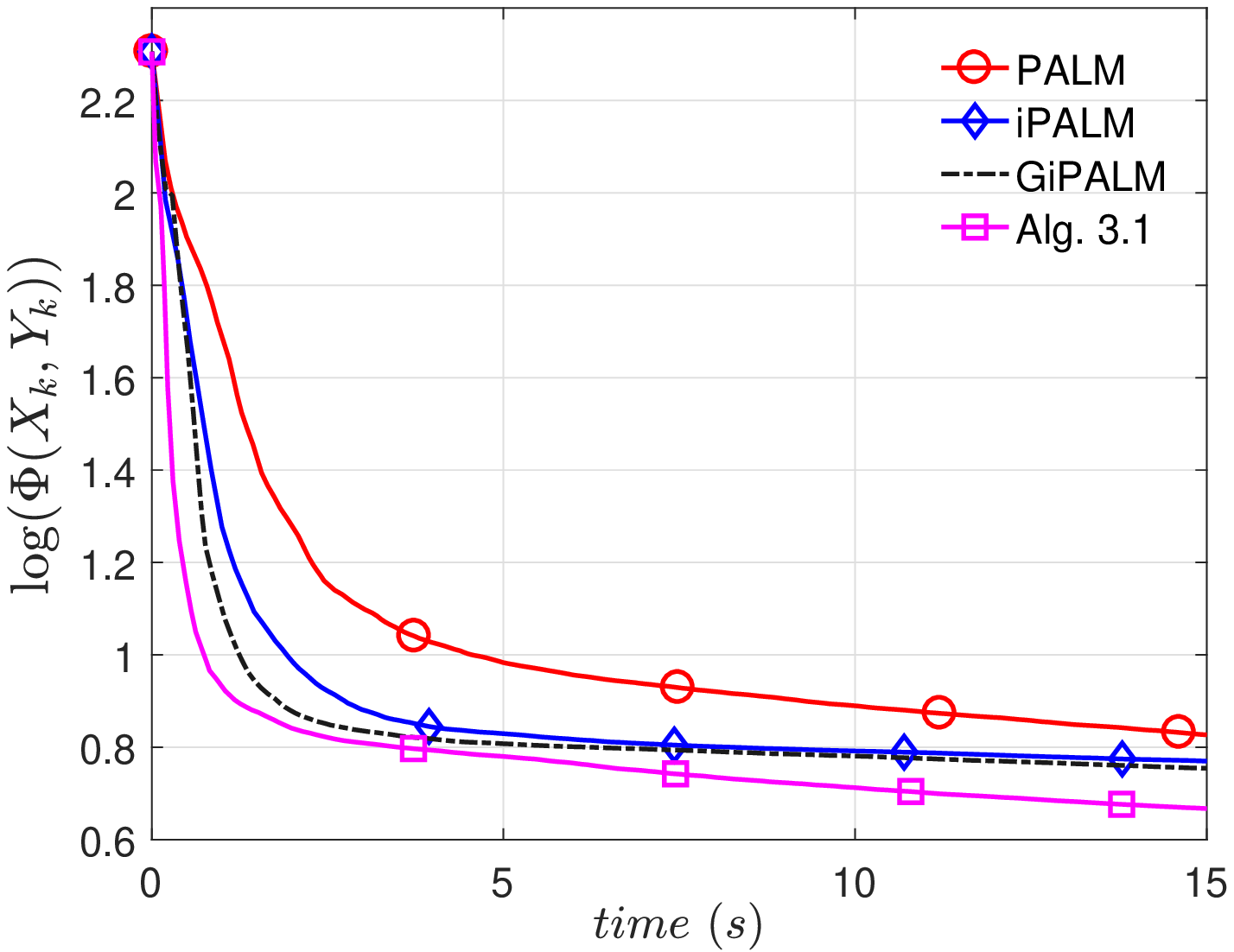}%
    \label{Wall-clock time comparison}}

    \caption{Objective decrease comparison of sparse-NMF with $s = 25\%$ on ORL dataset.}
    \label{fig_sim}
\end{figure*}

\begin{figure*}[!t]
    \centering
    \subfloat[Epoch counts comparison]{\includegraphics[width=3.0in]{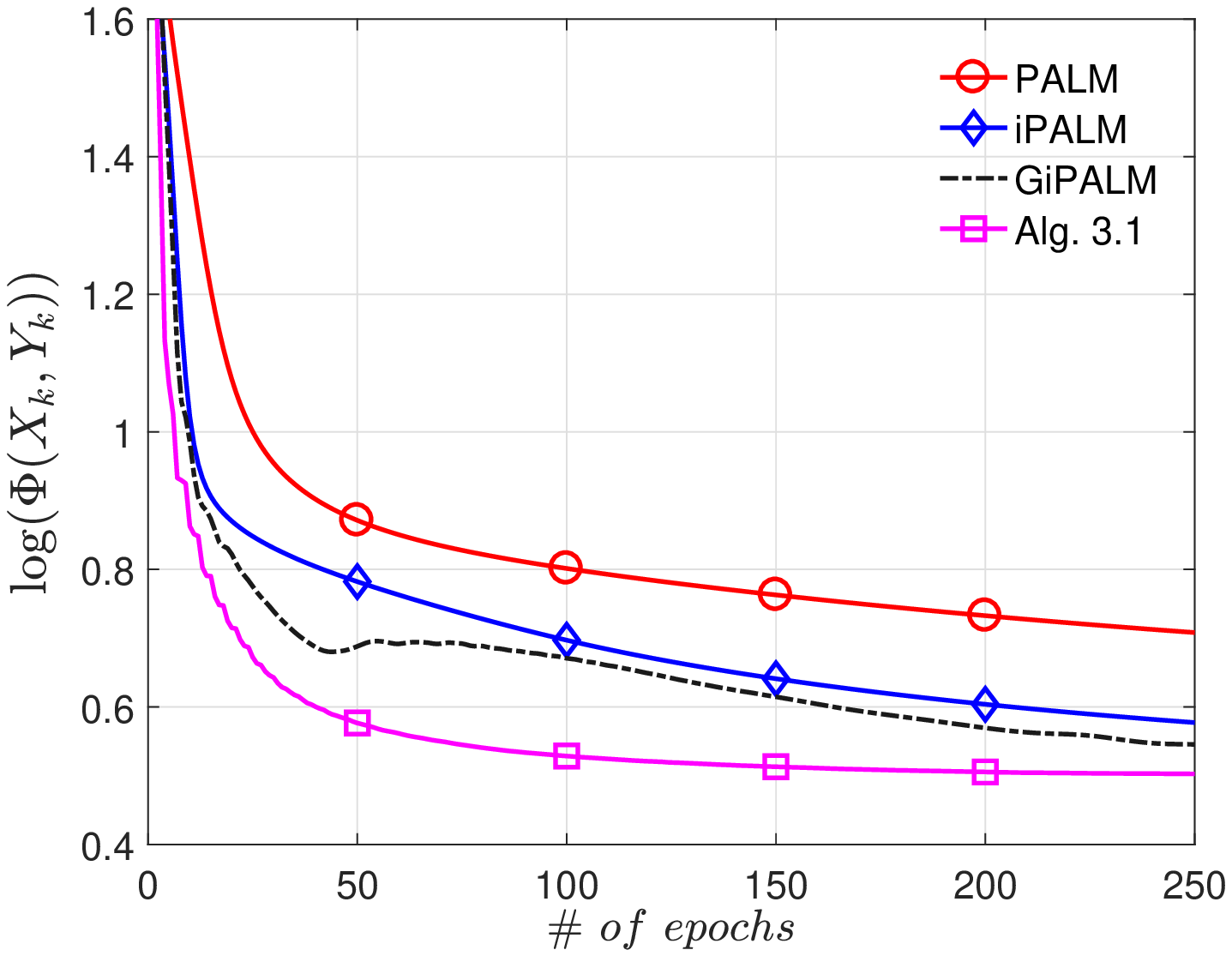}
    \label{Epoch counts comparison}}
\hfil
\subfloat[Wall-clock time comparison]{\includegraphics[width=3.0in]{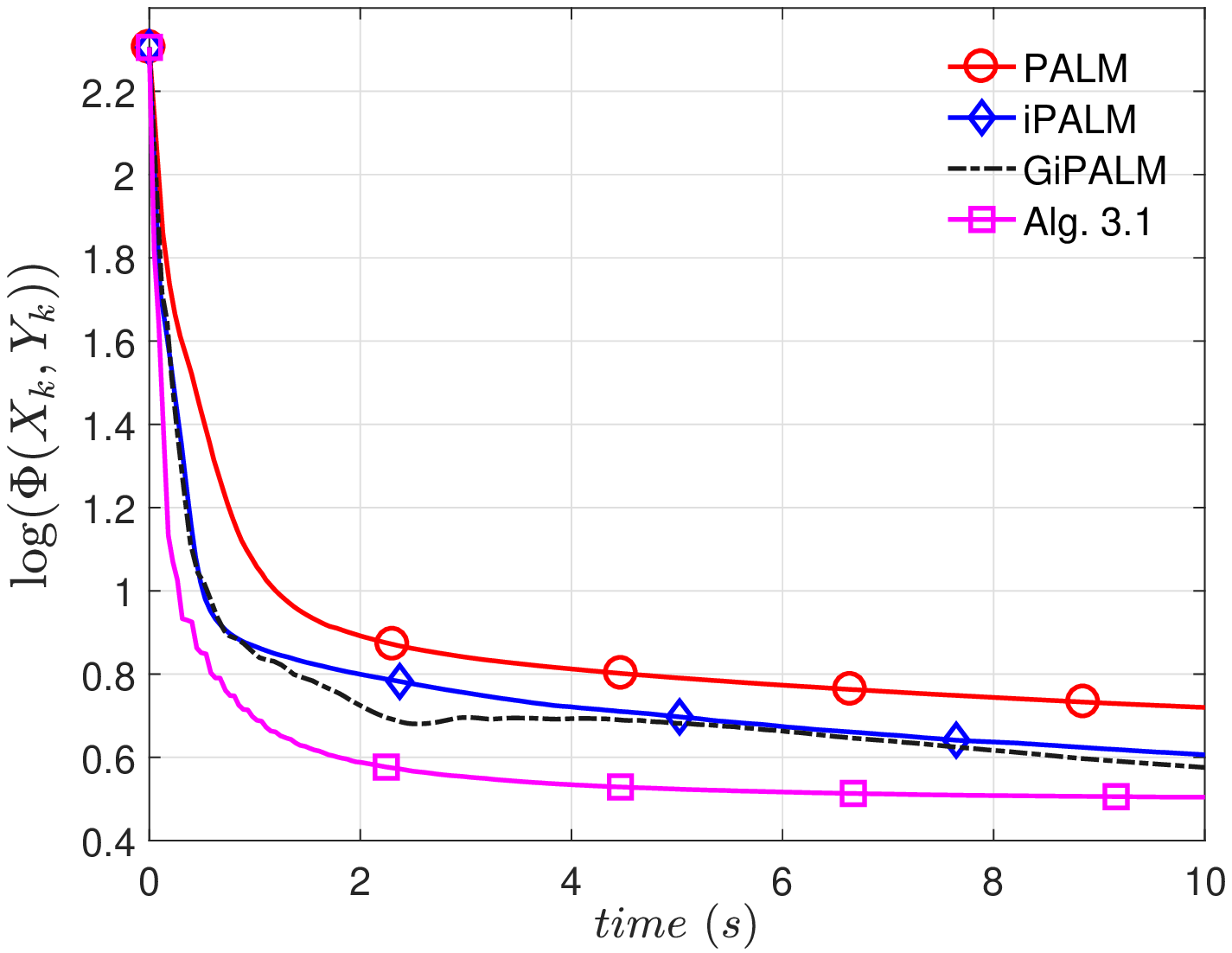}%
    \label{Wall-clock time comparison}}

    \caption{Objective decrease comparison of sparse-NMF with $s = 25\%$ on ORL dataset.}
    \label{fig_sim}
\end{figure*}

\begin{figure*}[!t]
    \centering
     \subfloat{\includegraphics[width=1.5in]{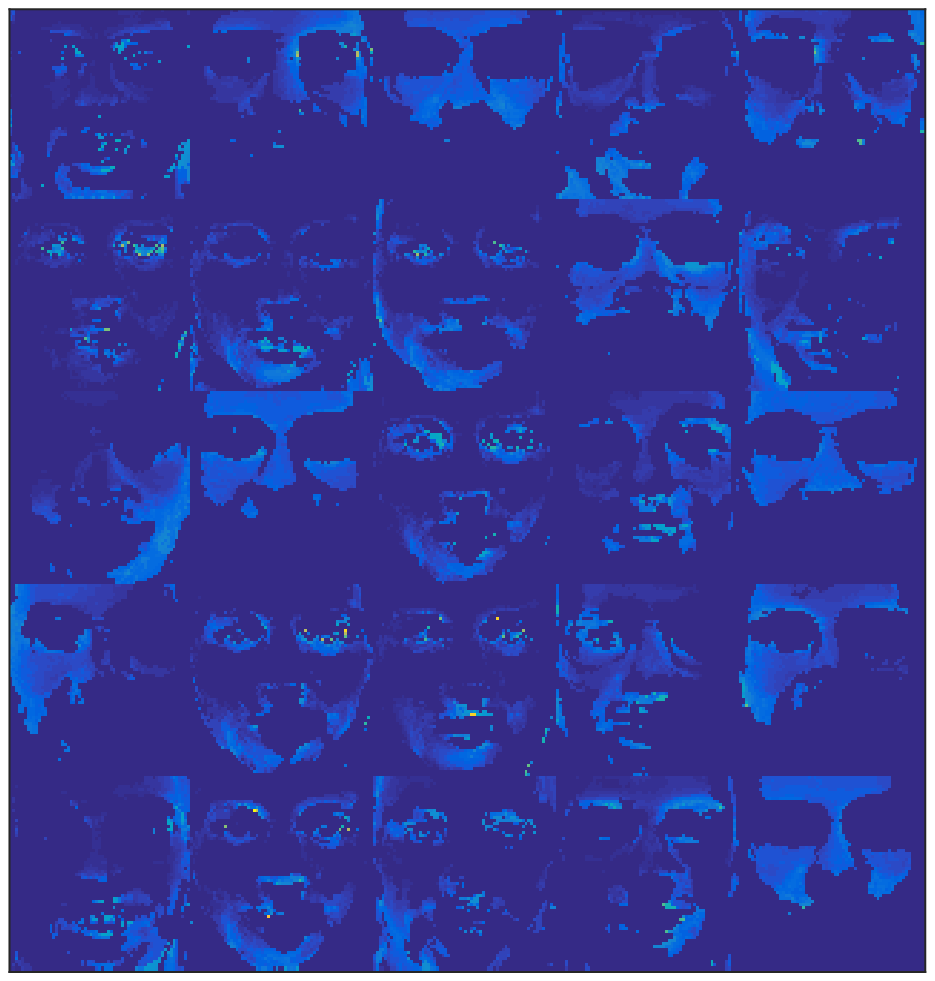}
    \label{fig_first_case}}
\hfil
\subfloat{\includegraphics[width=1.5in]{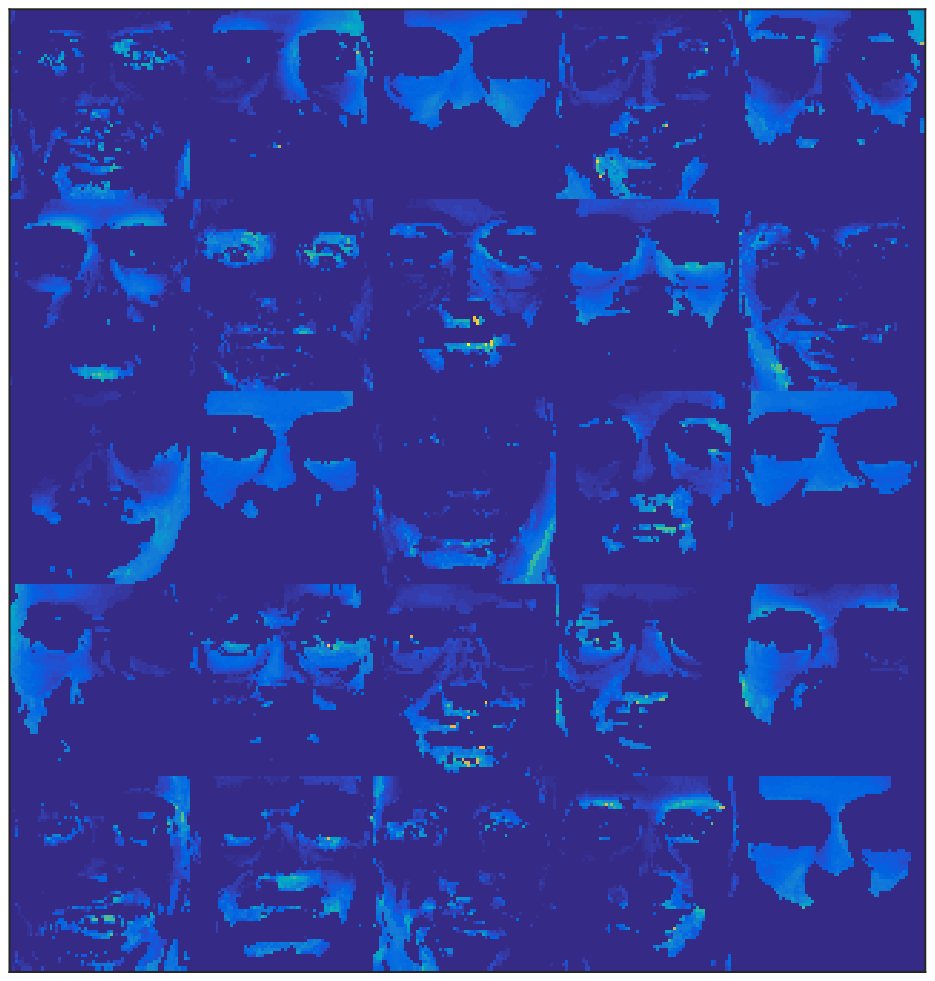}%
    \label{fig_second_case}}
\hfil
\subfloat{\includegraphics[width=1.5in]{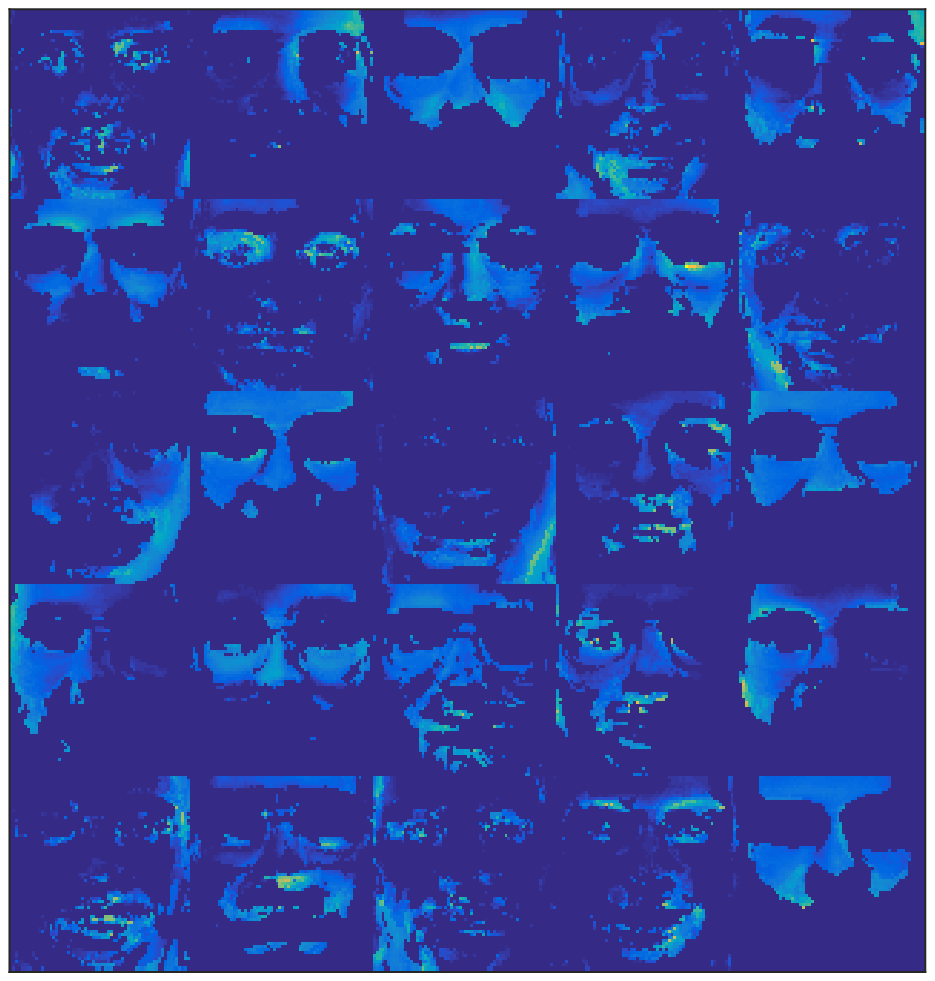}%
    \label{fig_thrid_case}}
\hfil    
    \subfloat{\includegraphics[width=1.5in]{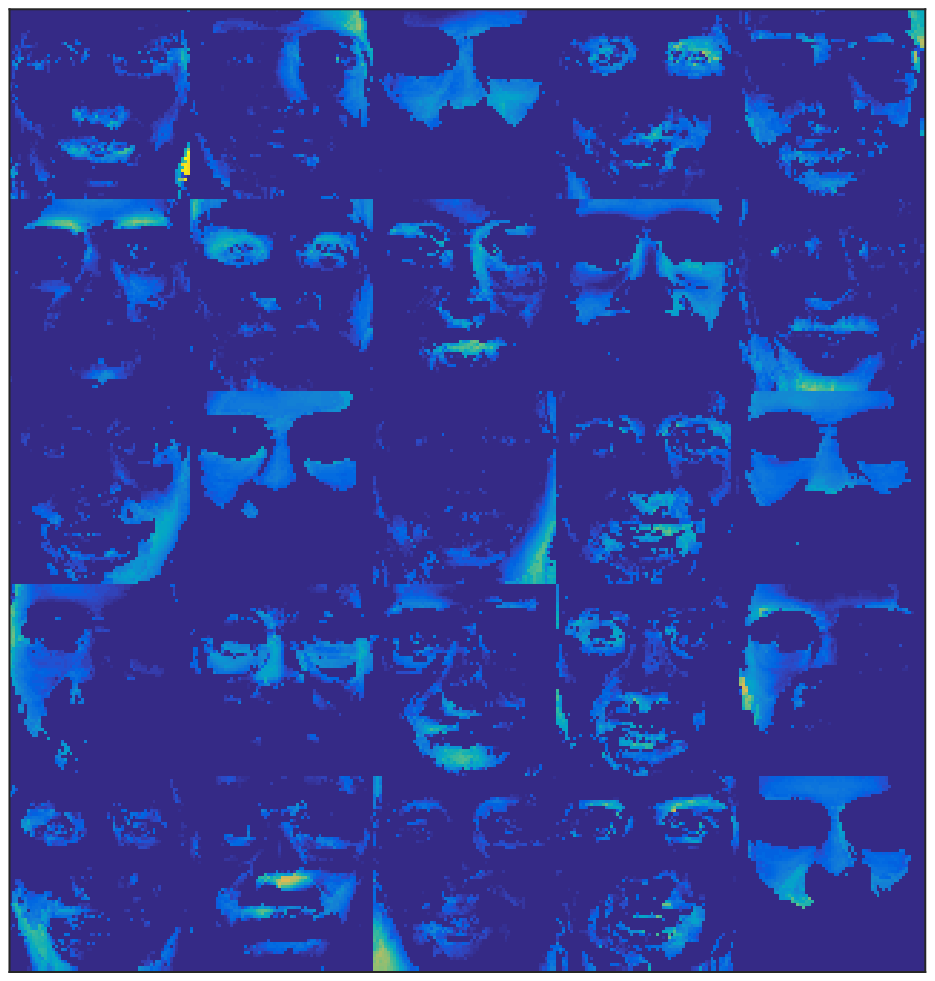}%
    \label{fig_four_case}}
\hfil
\subfloat{\includegraphics[width=1.5in]{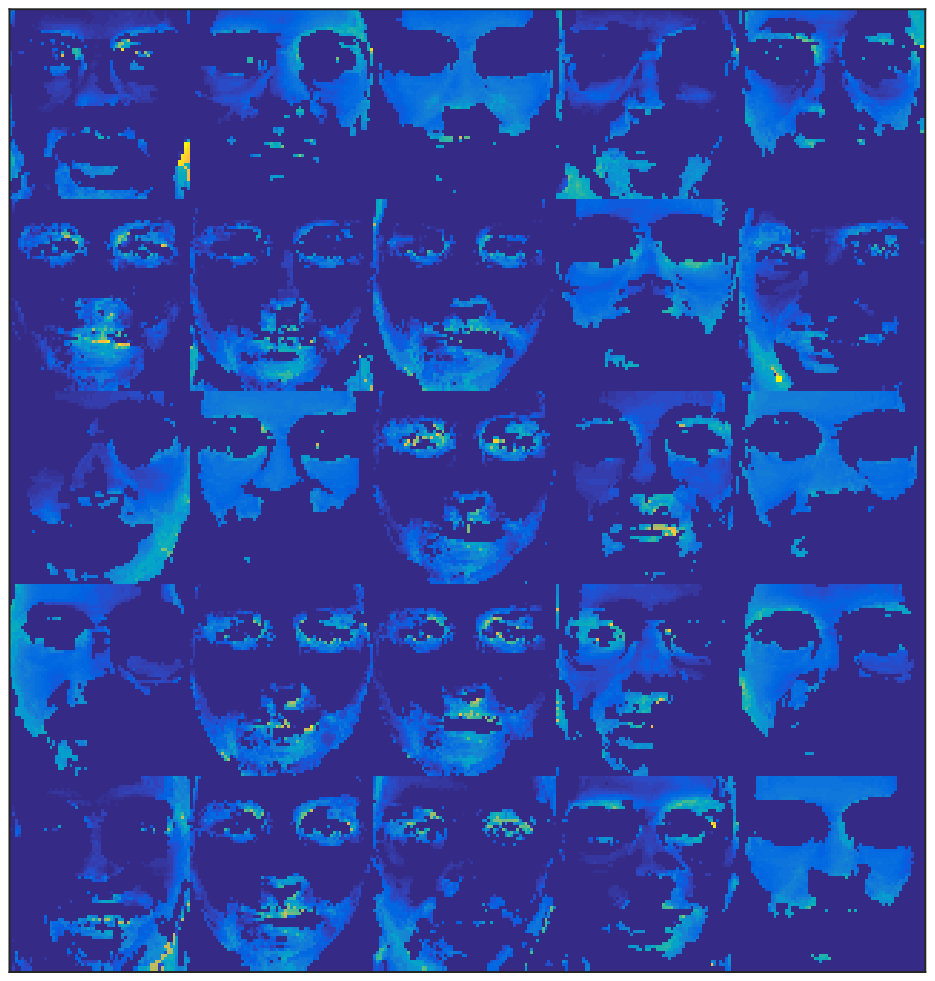}%
    \label{fig_five_case}}
\hfil
\subfloat{\includegraphics[width=1.5in]{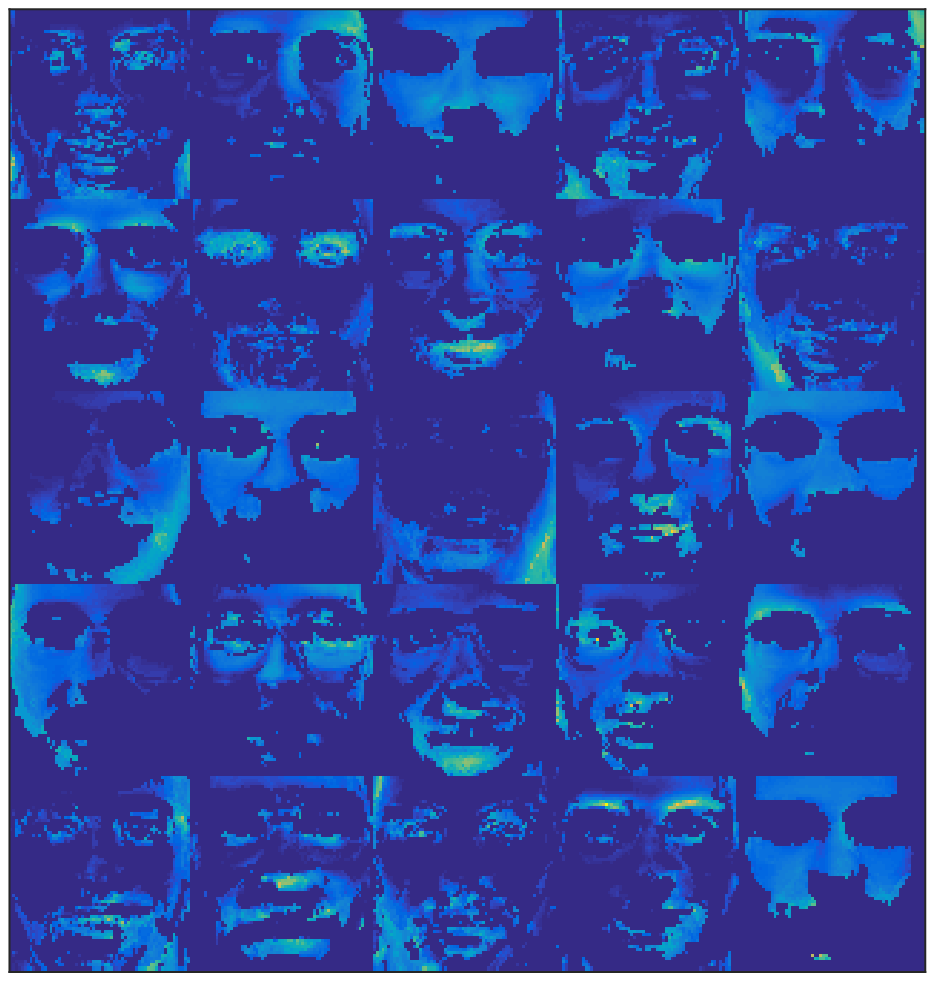}%
    \label{fig_six_case}}
    \hfil
\subfloat{\includegraphics[width=1.5in]{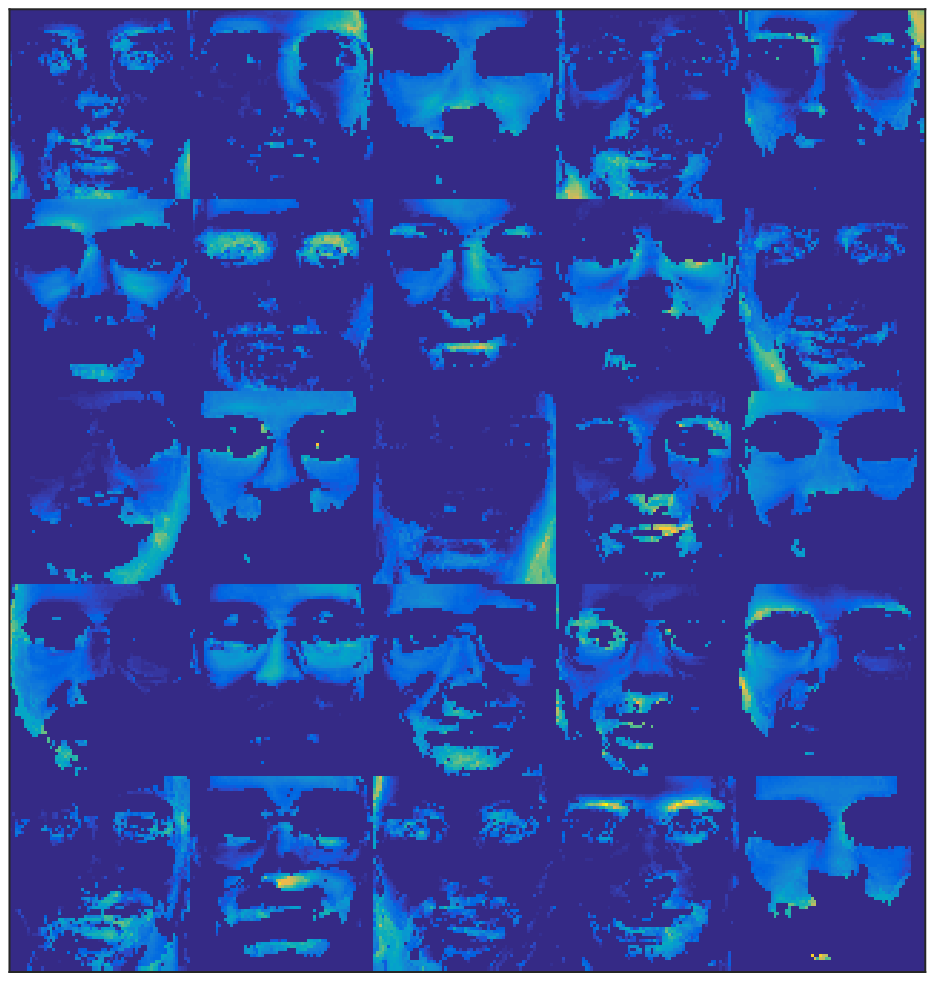}%
    \label{fig_seven_case}}
    \hfil
\subfloat{\includegraphics[width=1.5in]{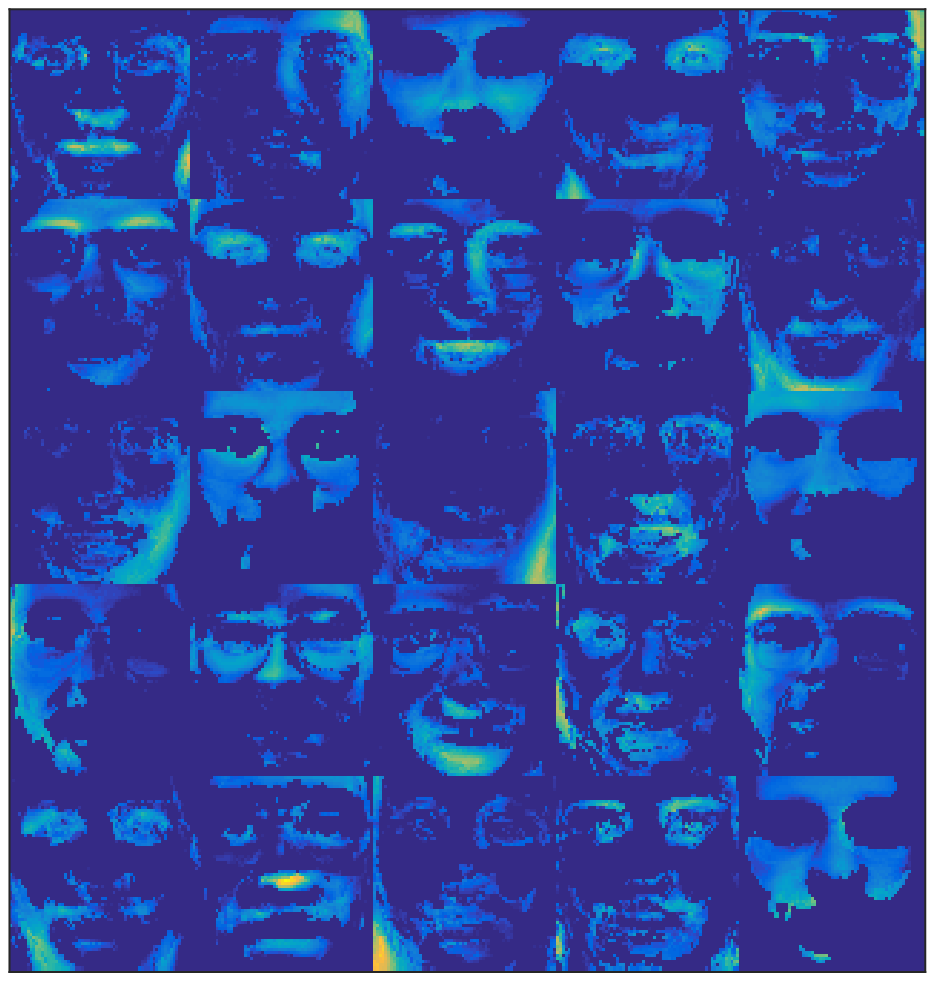}%
    \label{fig_eight_case}}   
    \hfil
\subfloat{\includegraphics[width=1.5in]{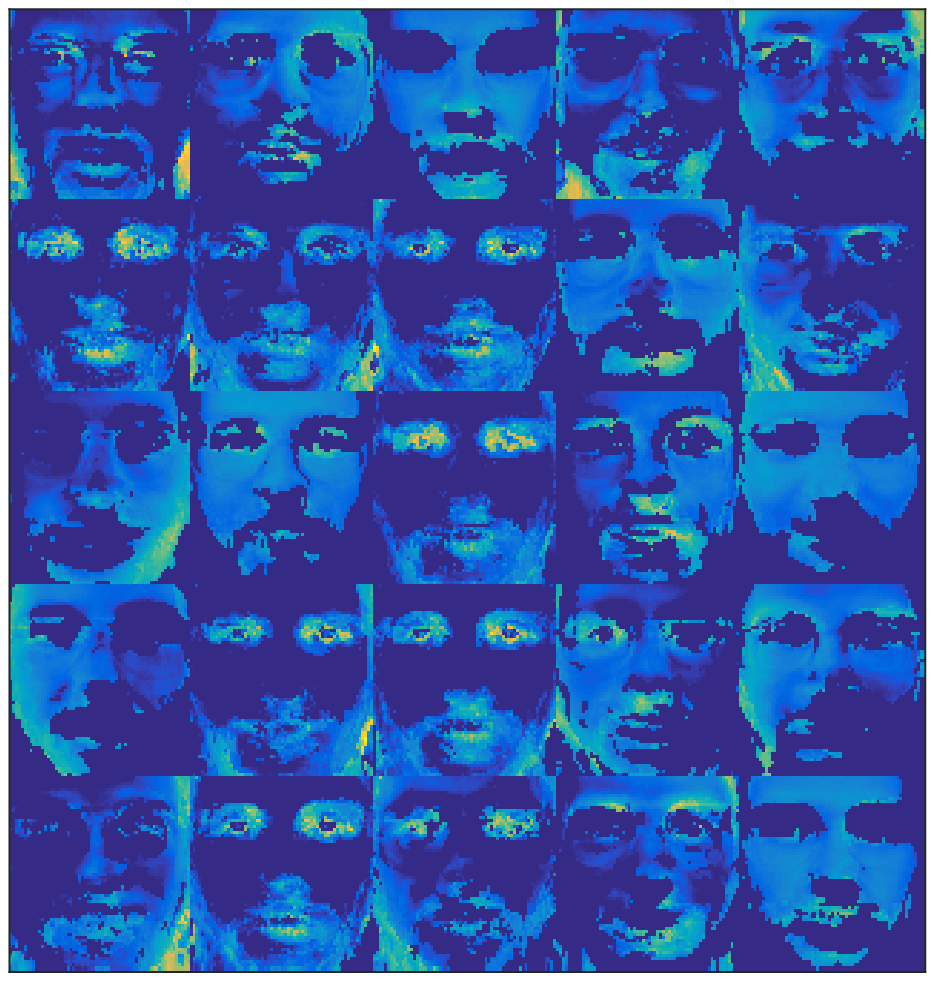}%
    \label{fig_nine_case}}
\hfil
\subfloat{\includegraphics[width=1.5in]{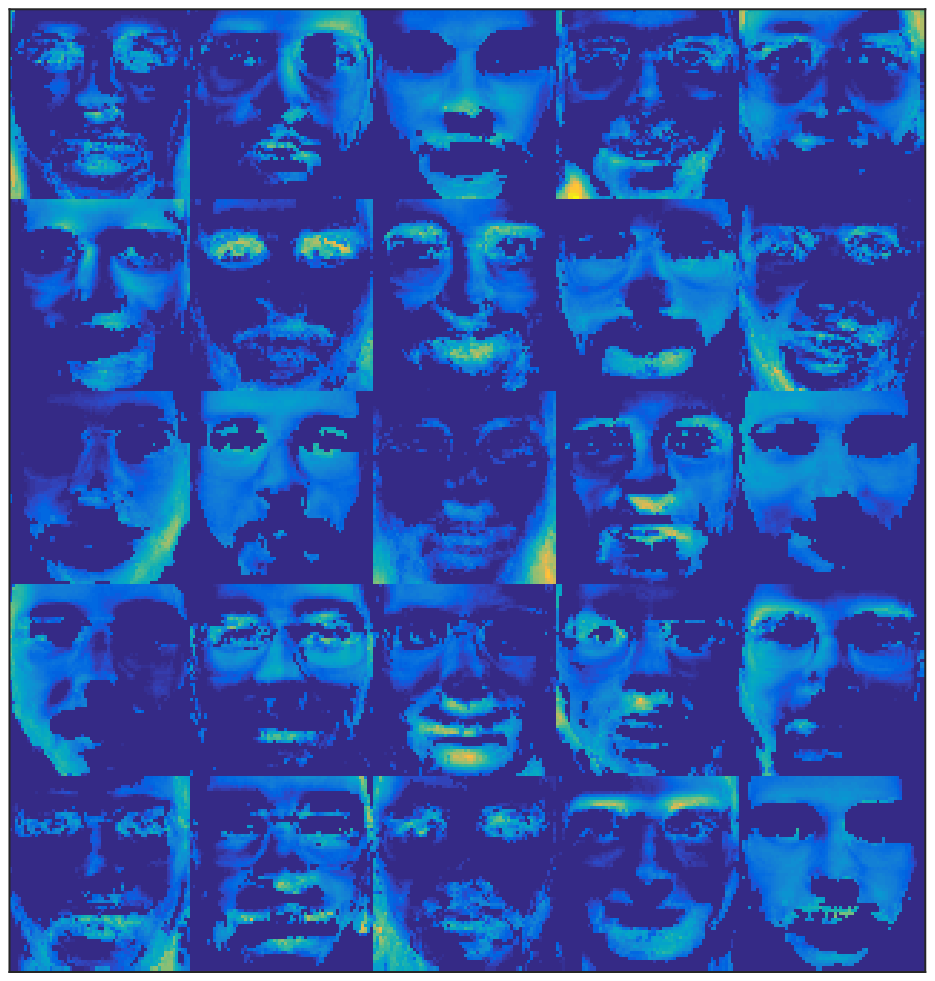}%
    \label{fig_ten_case}}
    \hfil
\subfloat{\includegraphics[width=1.5in]{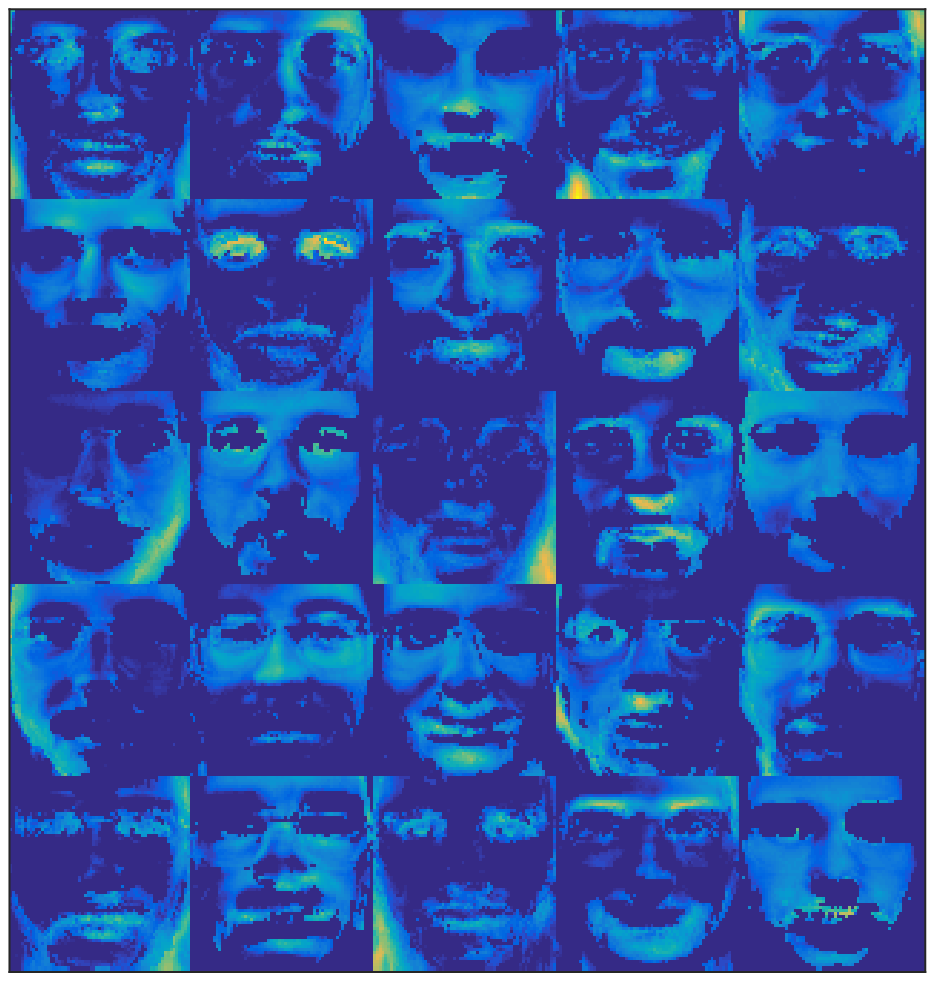}%
    \label{fig_eleven_case}}
    \hfil
\subfloat{\includegraphics[width=1.5in]{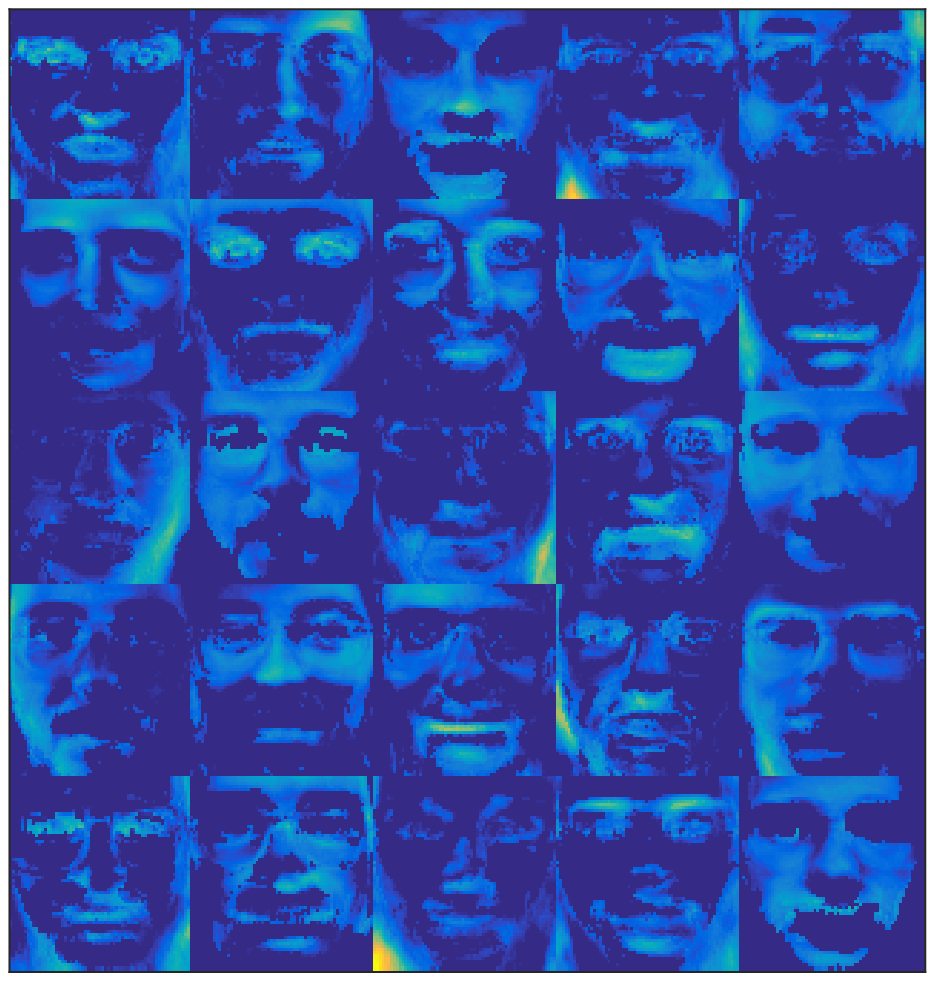}%
    \label{fig_twelve_case}} 
    \caption{The results for 25 basis faces using different sparsity settings. From left column to right column are the results of PALM, iPALM, GiPALM and Algorithm 3.1, respectively. From top row to bottom row are the result of $s = 25\%$  (sparsity of $s = 25\%$ means that each basis face contains $75\%$ nonzero pixels), $s = 33\%$ and $s = 50\%$, respectively. Clearly, stronger sparsity leads to a more compact representation.}
    \label{fig_sim2}
\end{figure*}
\subsection{Signal recovery}
\hspace*{\parindent}The sparse signal recovery problem has been studied extensively. Supposed $x$ is an unknown vector in $\mathbb{R}^m$ (a signal), given a matrix $A\in \mathbb{R }^{n\times m } $, and an observation $b\in \mathbb{R }^{n} $, we plan to recover $x$ from observation $b$ such that $x$ is of the sparsest structure (that is, $x$ has the fewest nonzero components). In this section, the sparse signal recovery problem can be modulated by the following $L_{0}$-problem\\
$$\min_{x}\\ {\rm }\ \ \left \| x \right \| _{0} $$
$$s.t.\\ {\rm }\ \ Ax=b.$$\\
The $L_{0}$-problem in the form of $L_{1/2} $ regularization can be described as \\
$$\min_{x}\\ {\rm }\ \  \frac{1}{2} \left \| Ax-b \right \| _{2}^{2} +\eta \left \| x \right \| _{1/2}^{1/2},\eqno{(4.2)} $$
where $\eta >0$ is used to balance regularization and data fitting. $\left \| x \right \| _{1/2} $ is the $L_{1/2}$ quasi-norm of $\mathbb{R }^{m}$, defined by $\left \| x \right \| _{1/2} =( {\textstyle \sum_{i=1}^{m}}\left | x_{i}  \right | ^{1/2}  )^{2} $. Now we illustrate how to implement the Algorithm 3.1 for solving the above model.\\
\hspace*{\parindent}By introducing an auxiliary variable $y\in \mathbb{R }^{m}$. The model (4.2) can be reformulated as \\
$$\min_{x,y}\\ {\rm }\ \ \frac{1}{2} \left \| Ax-b \right \| _{2}^{2} +\eta \left \| y \right \| _{1/2}^{1/2} ,$$
$$s.t.\\ {\rm }\ \ x=y,$$
which indeed can be described as\\
$$\min_{x,y}\\ {\rm }\ \  \frac{1}{2} \left \| Ax-b \right \| _{2}^{2} +\eta \left \| y \right \| _{1/2}^{1/2}+\frac{\gamma}{2}\left \| x-y \right \| _{2}^{2}.\eqno{(4.3)}$$
When we set $f(x)=\frac{1}{2} \left \| Ax-b \right \| _{2}^{2}, Q(x,y)=\frac{\gamma}{2}\left \| x-y \right \| _{2}^{2}, g(y)=\eta \left \| y \right \| _{1/2}^{1/2}$ , then (4.3) becomes nonconvex minimization problem (1.1). Let $\phi_1(x)= \langle x,Mx\rangle, \phi_2(y)=\frac{\lambda }{2} \left \| y\right \|_{2}^{2}$, then we can obtain colsed form of solution for solving the $x$-subproblem and $y$-subproblem of Algorithm 3.1 respectively as follows.\\
\hspace*{\parindent}Setting $M=\mu I- A^{T}A$, the $x$-subproblem corresponds to the following optimization problem
$$
\aligned
x_{k+1}
&&\in\arg\min_{ x\in \mathbb{R}^l}&\left \{\frac{1}{2} \left \| Ax-b \right \| _{2}^{2}+\langle x,\gamma(x_k-y_k)\rangle+\frac{1}{2}\left \| x-x_{k}  \right \| _{M}^{2} +\alpha_{1k} \langle x,x_{k-1}-x_k\rangle\right.\\
&&&\left.+\alpha_{2k} \langle x,x_{k-2}-x_{k-1}\rangle\right \}\\
&&=\arg\min_{ x\in \mathbb{R}^l}&\left \{\frac{1 }{2} \left \| Ax \right \| _{2}^{2}-\left \langle  Ax,b \right \rangle +\langle x,\gamma(x_k-y_k)\rangle+\frac{1}{2}\left \langle x-x_{k}  ,(\mu I- A^{T}A)(x-x_{k})  \right \rangle \right.\\
&&&\left.+\alpha_{1k} \langle x,x_{k-1}-x_k\rangle+\alpha_{2k} \langle x,x_{k-2}-x_{k-1}\rangle\right \}\\
&&=\arg\min_{ x\in \mathbb{R}^l}&\left \{\frac{\mu }{2} \left \| x \right \| _{2}^{2}-\langle x,\mu x_{k}- A^{T}Ax_{k}\rangle-\left \langle x,A^{T}b  \right \rangle  +\langle x,\gamma(x_k-y_k)\rangle\right.\\
&&&\left.+\alpha_{1k} \langle x,x_{k-1}-x_k\rangle+\alpha_{2k} \langle x,x_{k-2}-x_{k-1}\rangle\right \},\\
\endaligned
$$
which has an explicit expression
$$x_{k+1} =\frac{1}{\mu } \left [\mu x_{k} - A^{T}Ax_{k}+ A^{T}b-\gamma(x_k-y_k)+\alpha_{1k}(x_{k}-x_{k-1})+\alpha_{2k}(x_{k-1}-x_{k-2}) \right ] .$$
\hspace*{\parindent}The $y$-subproblem corresponds to the following optimization problem\\
$$
\aligned
y_{k+1}
&&\in \arg\min_{y\in \mathbb{R}^m}&\left \{\eta \left \| y \right \| _{1/2}^{1/2}+\langle y,\gamma(y_{k}-x_{k+1})\rangle+\frac{\lambda }{2} \left \| y-y_{k}  \right \| _{2}^{2}+\beta_{1k} \langle y,y_{k-1}-y_k\rangle\right.\\
&&&\left.+\beta_{2k} \langle y,y_{k-2}-y_{k-1}\rangle\right \}\\
&&= \arg\min_{y\in \mathbb{R}^m} &\left \{ \eta \left \| y \right \| _{1/2}^{1/2}+\frac{\lambda }{2}\left \| y-\frac{1}\lambda \left [ \lambda y_{k }+\gamma(x_{k+1}-y_{k})+\beta_{1k} (y_{k}-y_{k-1} )\right.\right.\right.\\
&&&\left.\left.\left.+\beta_{2k} (y_{k-1}-y_{k-2} ) \right ]  {}  \right \|   ^{2}\right \} \\
&&= \mathcal{H}&\left (y_{k}+\frac{1}\lambda \left [\gamma(x_{k+1}-y_{k})+\beta_{1k} (y_{k}-y_{k-1} )+\beta_{2k} (y_{k-1}-y_{k-2} ) \right ]   ,\frac{\eta }{\lambda } \right ),
\endaligned
$$
where, for any $\kappa >0$, $\mathcal{H}(\cdot ,\kappa )$ is called the half shrinkage operator \cite{XCX} defined as\\
$${\mathcal{H}\rm}(a;\kappa )=\left \{ h_{\kappa } (a_{1} ) ,h_{\kappa } (a_{2} ),\dots ,h_{\kappa } (a_{n} )\right \} ^{T} $$
with
$$h_{\kappa } (a_{i{\tiny } } )=\left\{\begin{matrix}
  \frac{2a_{i} }{3} (1+cos(\frac{2\pi}{3} -\frac{2}{3}\varphi (a_{i})) ), & \left | a_{i}  \right |  >\frac{3}{4} \kappa ^{2/3}, \\
  0, &otherwise
\end{matrix}\right.$$
and $\varphi (a_{i} )=$arccos$\left ( \frac{\kappa }{8} \left ( \frac{\left | a_{i}  \right | }{3}  \right )^{-3/2}  \right ) .$\\

In this experiment, each entry of $A$ is drawn from the standard normal distribution, and then all columns of $A$ are normalized according to $L_{1/2}$ quasi-norm, noted that $\left \| A\right \|\le1$; we genergte a random sparse vector as $x$; the noise vector $\omega \sim N\left ( 0,10^{-3}I  \right )$ and the vector $b=Ax+\omega $; the regularization parameter $\eta=0.001\left \| A^{T}b  \right \|_{\infty }$ and the penalization parameter is set as $\gamma  =0.2$. By Lemma \ref{lem31}, we take $\alpha_1>0, \alpha_2>0$ such that $2(\alpha_1+\alpha_2)<\rho $, where $\rho =\min\{\mu - \left \| A \right \| ^{2}  -\gamma, \lambda  -\gamma\}$, $\mu =2,\lambda =1.5$. In Algorithm 3.1 and TiBAM with two-step inertial extrapolation, we select inertial parameters $\alpha_{1k}=\alpha_{2k}=\beta _{1k}=\beta _{2k}=0.99\rho/4$; in iBPALM with one-step inertial extrapolation, we select inertial parameters $\alpha_{1k}=\beta _{1k}=0.99\rho/2$; BPALM is the special case of Algorithm 3.1 without inertial extrapolation, that is, $\alpha_{1k}=\alpha_{2k}=\beta _{1k}=\beta _{2k}=0$; 
We take the origin point as the initial point for all algorithms and use
 $$E_k=\|x_{k+1}-x_k\|+\|y_{k+1}-y_k\|<10^{-4}$$ as the stopping criterion. In the following, we compare Algorithm 3.1 with TiBAM, iBPALM and BPALM for matrix $A$ of different dimensions with $(n,m)=(40,200)$ and $(n,m)=(100,500)$.

\begin{table}[htbp]
\caption{Compare Algorithm 3.1 with TiBAM, iBPALM and BPALM for $n = 40, m=200$. }\label{table1}
\begin{tabular}{llllllll}
\hline
Algorithm  &       &$b=Ax$   &                                &&       &$b=Ax+\omega$& \\ \cline{2-4} \cline{6-8}
n=40,m=200 & Iter.  & Time(s) &$\left \|x_{k}-y_k  \right \| $ && Iter.  & Time(s)  & $\left \|x_{k}-y_k  \right \|  $ \\ \hline
Alg. 3.1   & 713   & 4.1934  &$5.0602\times 10^{-5}$          && 810   & 6.5243   & $6.5298\times 10^{-5}$        \\
TiBAM     & 1110  & 7.6269  &$8.9963\times 10^{-5}$          && 1230  & 9.9036   & $9.1096\times 10^{-5}$  \\
iBPALM      & 1378  & 9.8594  &$9.3004\times 10^{-5}$          && 1577  & 12.0625  & $1.0239\times 10^{-4}$ \\ 
BPALM       & 2033  & 12.8073 &$1.7996\times 10^{-4}$          && 2276  & 14.7630  & $1.9939\times 10^{-4}$ \\ \hline
\end{tabular}
\end{table}

\begin{figure*}[!t]
    \centering
    \subfloat[$b=Ax$]{\includegraphics[width=3.0in]{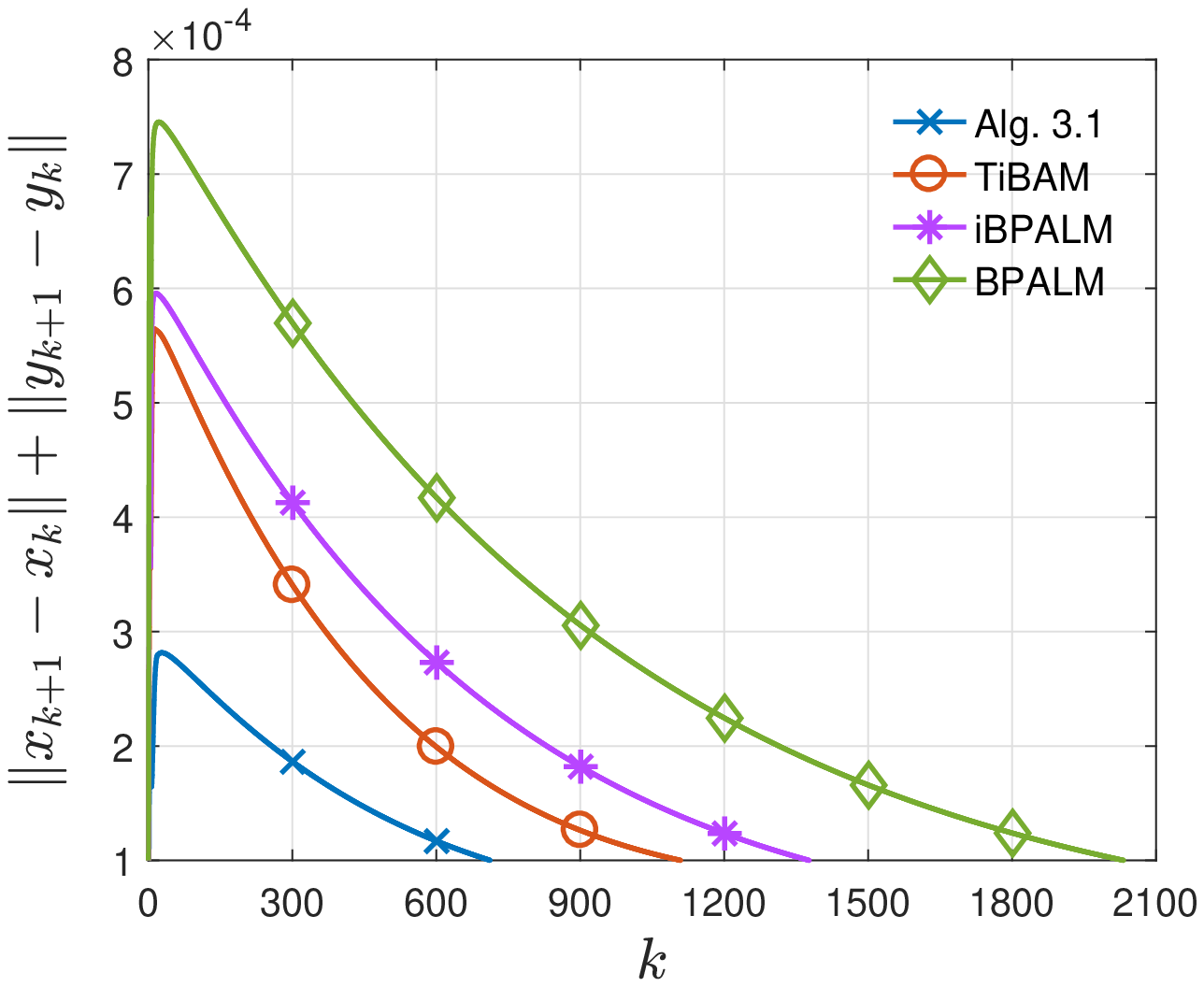}
    \label{fig_first_case}}
\hfil
\subfloat[$b=Ax+\omega$]{\includegraphics[width=3.0in]{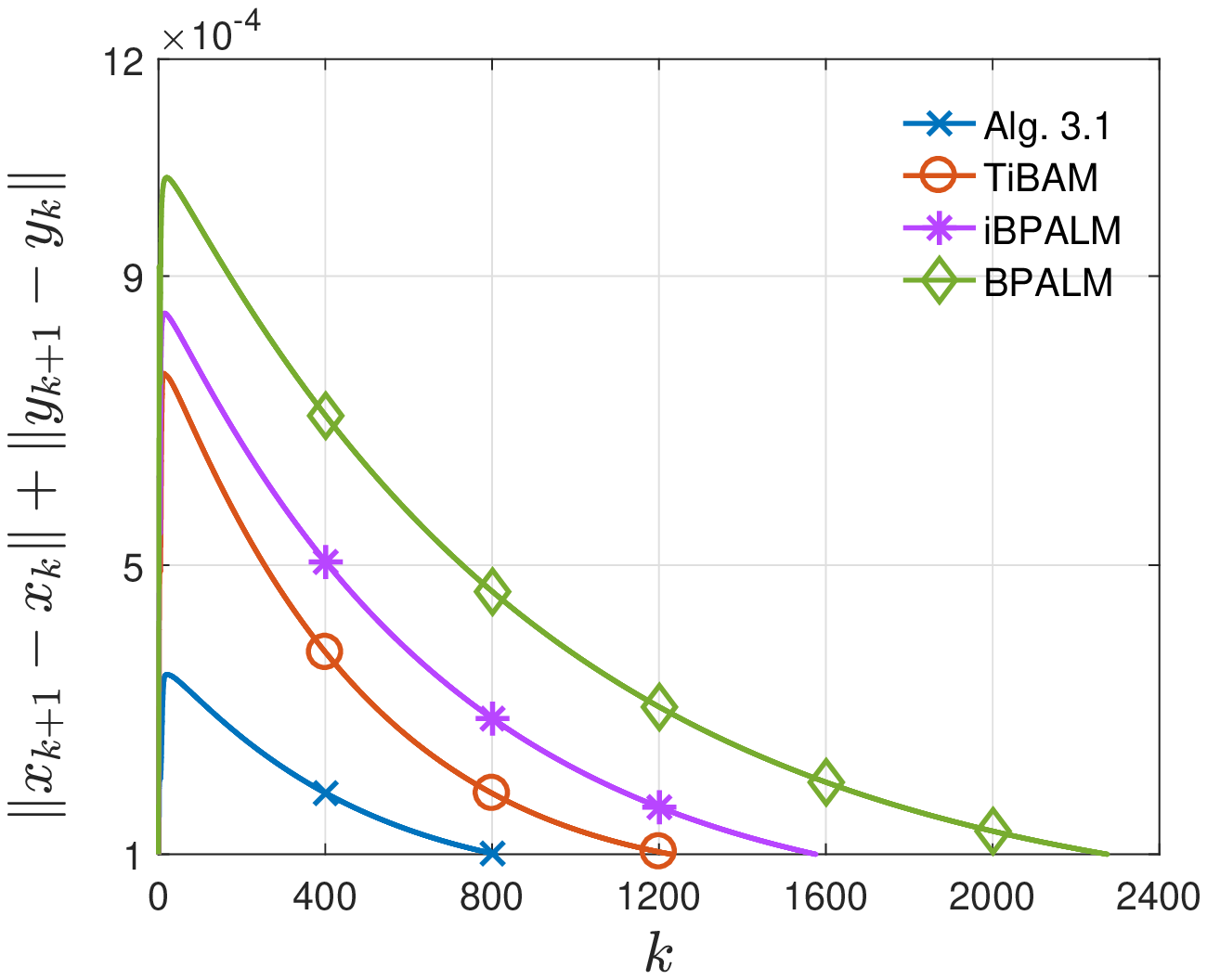}%
    \label{fig_second_case}}

    \caption{The value of $\|x_{k+1}-x_k\|+\|y_{k+1}-y_k\|$ versus the iteration numbers for 4.2 with $n = 40, m=200$.}
    \label{fig_sim}
\end{figure*}
Table 1 and Table 2 report the number of iteration, CPU time and the 2-norm of $x_k-y_k$ for $(n,m)=(40,200)$ and $(n,m)=(100,500)$, respectively. In the numerical results, ``Iter." denotes  the number of iterations. ``Time" denotes  the CPU time. It is obvious that Algorithm 3.1 has an advantage over TiBAM, iBPALM and BPALM in terms of iteration number and time for solving the above problem. In Figure 7 and Figure 8, the left picture records the downward trend of the error function without noise $(b=Ax)$ and the right picture shows the trend of the error function with noise $(b=Ax+\omega)$ for $(n,m)=(40,200)$ and $(n,m)=(100,500)$, respectively. They show the efficiency and advantage of Algorithm 3.1.

\begin{table}[htbp]
\caption{Compare Algorithm 3.1 with TiBAM, iBPALM and BPALM for $n = 100, m=500$.}\label{table2}
\begin{tabular}{llllllll}
\hline
Algorithm  &       &$b=Ax$   &                                &&       &$b=Ax+\omega$& \\ \cline{2-4} \cline{6-8}
n=100,m=500 & Iter.  & Time(s) &$\left \|x_{k}-y_k  \right \| $ && Iter.  & Time(s)  & $\left \|x_{k}-y_k  \right \|  $ \\ \hline
Alg. 3.1   & 1610  & 40.3756  &$9.9978\times 10^{-5}$          && 1920  & 46.5243   & $1.5527\times 10^{-4}$        \\
TiBAM     & 2108  & 48.5674  &$1.3827\times 10^{-4}$          && 2467  & 52.6250   & $1.8418\times 10^{-4}$  \\
iBPALM      & 2732  & 55.7853  &$2.0018\times 10^{-4}$          && 3196  & 59.1563   & $2.1329\times 10^{-4}$ \\ 
BPALM      & 3731  & 65.3309  &$2.9998\times 10^{-4}$          && 4023  & 69.7543   & $3.0698\times 10^{-4}$ \\ \hline
\end{tabular}
\end{table}

\begin{figure*}[!t]
    \centering
    \subfloat[$b=Ax$]{\includegraphics[width=3.0in]{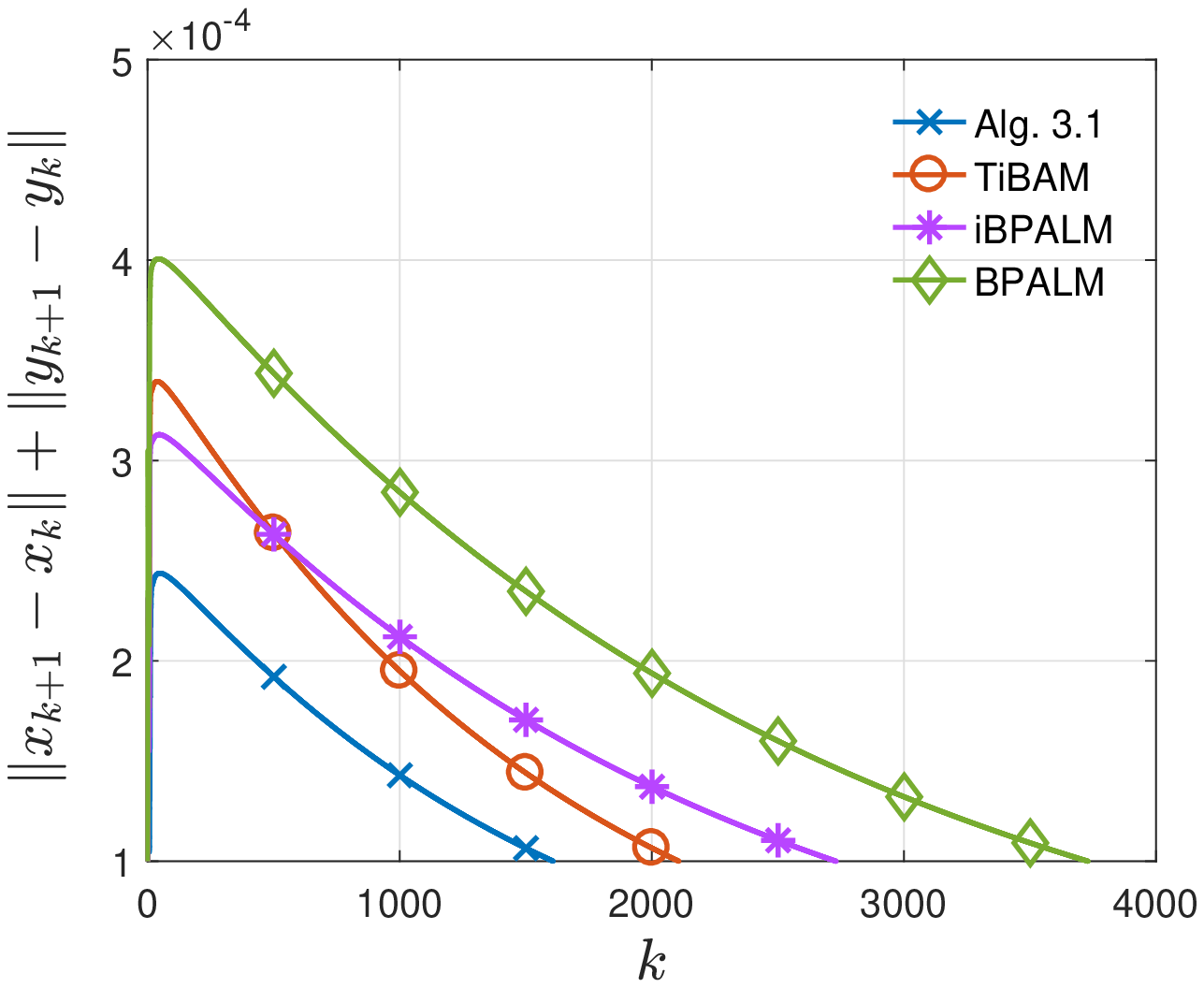}
    \label{fig_first_case}}
\hfil
\subfloat[$b=Ax+\omega$]{\includegraphics[width=3.0in]{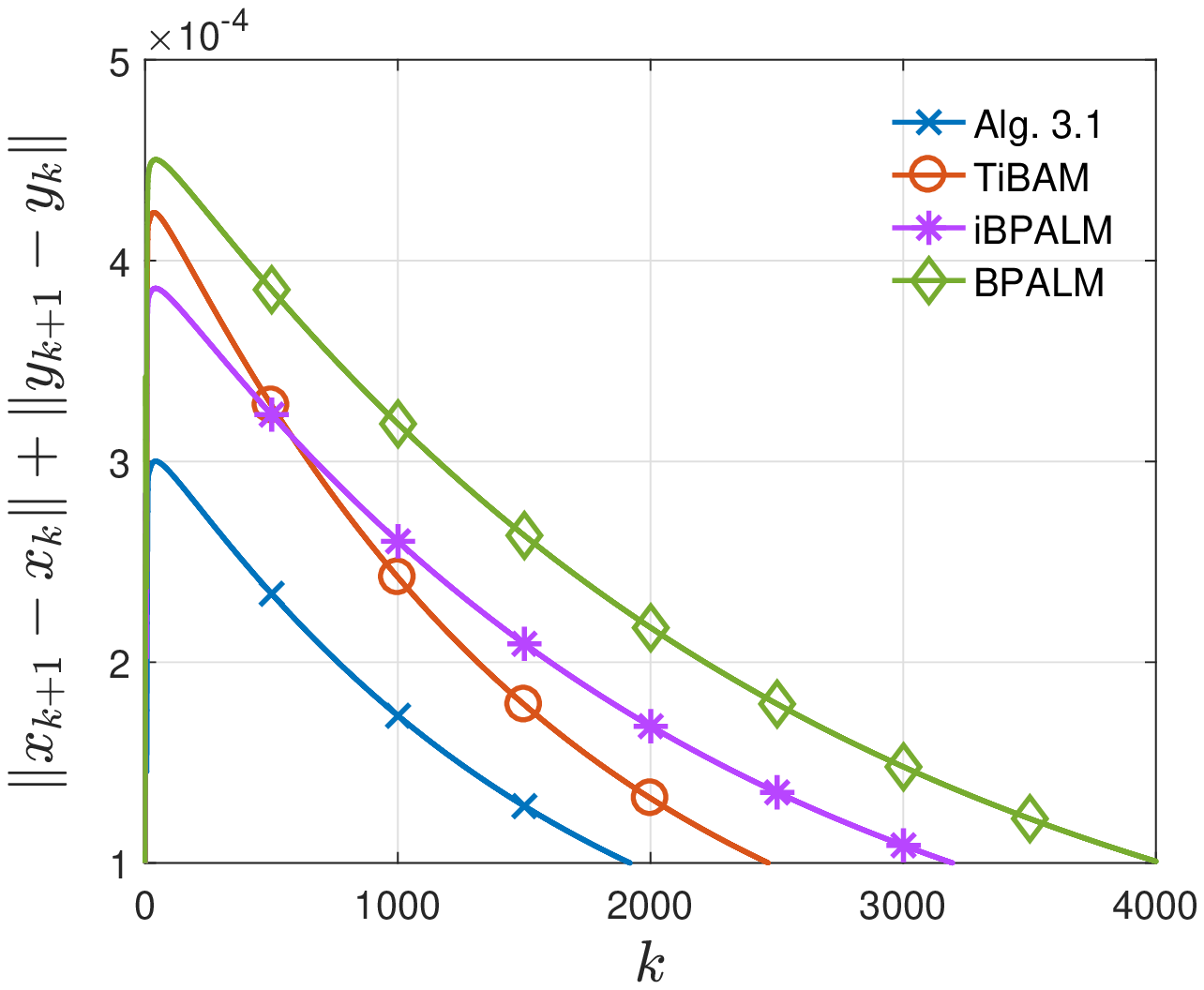}%
    \label{fig_second_case}}

    \caption{The value of $\|x_{k+1}-x_k\|+\|y_{k+1}-y_k\|$ versus the iteration numbers for 4.2 with $n = 100, m=500$.}
    \label{fig_sim}
\end{figure*}

\subsection{Nonconvex quadratic fractional programming}
\hspace*{\parindent}In this subsection, we provide some numerical experiments which we carried out in order to illustrate the numerical convergence of Algorithm 3.1 with different Bregman distances.
The following list are various functions with its Bregman distances:

(i) Define the function $\varphi_1(x)=\mu\sum_{i=1}^m x_i\ln x_i$ with domain

$$\text{dom}\varphi_1=\{x=(x_1, x_2,\cdots, x_m)^T\in \mathbb{R}^m: x_i > 0, i =1, 2,\cdots, m\}$$
 and range ran$\varphi_1=(-\infty,+\infty)$. Then
$$\nabla \varphi_1(x)=\mu(1+\ln(x_1), 1+\ln(x_2), \cdots, 1+\ln(x_m))^T$$
and the Bregman distance (the Kullback-Leibler distance) with respect to $\varphi_1$ is
$$D_{\varphi_1}(x, y) = \mu\sum_{i=1}^m\big(x_i\ln\big(\frac{x_i}{y_i}\big)+ y_i-x_i\big),\ \ \forall x,y\in \mathbb{R}_{++}^m.$$
\hspace*{\parindent}(ii) Define the function $\varphi_2(x)=-\mu\sum_{i=1}^m \ln x_i$ with domain

$$\text{dom}\varphi_2=\{x=(x_1, x_2,\cdots, x_m)^T\in \mathbb{R}^m: x_i > 0, i =1, 2,\cdots, m\}$$
 and range ran$\varphi_2=(-\infty,+\infty)$. Then
$$\nabla \varphi_2(x)=\mu(-\frac{1}{x_1}, -\frac{1}{x_2}, \cdots, -\frac{1}{x_m})^T$$
and the Bregman distance (the Itakura-Saito distance) with respect to $\varphi_2 $ is
$$D_{\varphi_2}(x, y) = \mu\sum_{i=1}^m\big(\frac{x_i}{y_i}-\ln\big(\frac{x_i}{y_i}\big)-1\big),\ \ \forall x,y\in \mathbb{R}_{++}^m.$$
\hspace*{\parindent}(iii) Define the function $\varphi_3(x)=\frac{\mu}{2}\|x\|^2$ with domain $\text{dom}\varphi_3=\mathbb{R}^m$
 and range ran$\varphi_3=[0,+\infty)$. Then $\nabla \varphi_3(x)=x$ and the Bregman distance (the squared Euclidean distance) with respect to $\varphi_3$ is
$$D_{\varphi_3}(x, y) = \frac{\mu}{2}\|x-y\|^2,\ \ \forall x,y\in \mathbb{R}^m.$$
It  is clear that $\varphi_i$ is $1$-strongly convex ($i=1,2,3$).\\

We consider the quadratic fractional programming problem \\
$$\min_{x\in C}\\ f(x):=\frac{x^{T}Mx+a^{T}x+c}{b^{T}x+d}  $$
with\\
$$C=\left \{ x\in \mathbb{R}^m :1\le x_{i}\le 3,i=1,2,\cdots,m  \right \} ,$$
where $M:\mathbb{R}^m\to  \mathbb{R}^m$ is a bounded linear operator, $a\in \mathbb{R}^m$, $b\in \mathbb{R}^m$, $c=-2$ and $d=20$. By \cite{BCV}, we know $f$ is pseudo-convex on the open set $E=\left \{ x\in \mathbb R^{m} :b^{T}x+d >0  \right \}$, if $C\subseteq E$, then $f$ is nonconvex.\\
\hspace*{\parindent}The quadratic fractional programming problem can be rewritten as (\ref{MP}):
$$\min\{f(x)+\frac{\gamma}{2}\|x-y\|^2_2+\iota_C(y):x,\ y\in\mathbb{R}^{m}\},$$
where $\iota_C$ is the indicator function on $C$.\\
\hspace*{\parindent}We now elaborate the $x$-subproblem and $y$-subproblem of Algorithm 3.1 respectively as follows.\\
$$
\begin{cases}
\aligned
x_{k+1}\in \arg\min_{ x\in \mathbb{R}^m}\{&f(x)+\langle x,\gamma(x_{k}-y_{k})\rangle+D_{\phi_1}(x,x_k)+\alpha_{1k} \langle x,x_{k-1}-x_k\rangle\\
&+\alpha_{2k} \langle x,x_{k-2}-x_{k-1}\rangle\},\\
y_{k+1}\in \arg\min_{ y\in \mathbb{R}^m}\{&\iota_C(y)+\langle y,\gamma(y_{k}-x_{k+1})\rangle+D_{\phi_2}(y,y_k)+\beta_{1k} \langle y,y_{k-1}-y_k\rangle\\
&+\beta_{2k} \langle y,y_{k-2}-y_{k-1}\rangle\}.
\endaligned
\end{cases}
$$

In this experiment, we take $\alpha_{1k}=\beta_{1k}=0.2$, $\alpha_{2k}=\beta_{2k}=0.3, \gamma=10,  \mu=36$. We perform the numerical tests of Algorithms 3.1 with different Bregman distance and use
 $$E_k=\|x_{k+1}-x_k\|+\|y_{k+1}-y_k\|<10^{-4}$$ as the stopping criterion. We use ``Alg(ij)" to denote Algorithm 3.1 with $\phi_1(x)=\varphi_i(x)$ and
 $\phi_2(x)=\varphi_j(x)$ $(1\leq i\leq 3, 1\leq j\leq 3)$. For the above quadratic fractional programming problem, we will give the numerical results of Algorithm 3.1 for different matrix $M$ and dimensions. We randomly selected the starting point, carried out 30 times randomly, and obtained the averaged iteration number and the averaged CPU time. 
The numerical results for the performance of Algorithm 3.1 with different Bregman distance are shown in Table 3 and Table 4. In the numerical results, ``Iter." denotes  the number of outer iterations. ``InIterx." and ``InItery." denote, respectively, the number of inner iterations for solving $x-$subproblem and $y-$subproblem. ``Time" denotes  the CPU time.

\noindent{\bf Problem 1}. We consider a fixed matrix $M=\begin{bmatrix}
  5&  -1&  2&  0& 2\\
  -1&  6&  -1&  3& 0\\
  2&  -1&  3&  0& 1\\
  0&  3&  0&  5& 0\\
  2&  0&  1&  0&4
\end{bmatrix}$, and give the vectors $a=(1,2,-1,-2,1)^{T},b=(1,0,-1,0,1)^{T}$, in this case, $C\subseteq E$. For different bregman distances, the numerical results are given for one-step inertial and two-step inertial extrapolation in Table 3 below, where the one-step inertial extrapolation $\alpha_{1k}=\beta_{1k}=0.5$.

\begin{table}[h]\caption{Numerical results of Algorithm 3.1 for fixed matrix $M$}\label{table3}
\begin{center}
{\scriptsize
\begin{tabular} {c c c c c c c c c c}
\hline &Alg.(11)&Alg.(12)&Alg.(13)&Alg.(21)&Alg.(22)&Alg.(23)&Alg.(31)&Alg.(32)&Alg.(33)\\
\hline
$one-step$&&&&&&&&&\\
Iter.&662    &675  &657    &307   &308   &305   &3649   &3497  &3918\\ 
InIterx.&992 &1003   &988    &493   &497 &490    &3812   &3724   &4090\\
InItery.&704  &719  &699    &307   &310  &305    &3907   &3731   &4237\\
Time(s)&0.4063&0.3750 &0.2031  &0.1205 &0.2188  &0.0718  &0.9188 &0.8563 &1.1938\\
\hline
$two-step$&&&&&&&&&\\
Iter.&529     &534   &523  &294   &295   &293   &2984  &2828  &3513\\ 
InIterx.&830 &834   &825   &493  &495  &491  &3237  &3114  &3698\\
InItery.&559 &565   &552   &294   &296   &292   &3193  &3019  &3765\\
Time(s)&0.2813&0.3281&0.1563  &0.0938&0.1938&0.0681 &0.8063&0.7501&1.0531\\
\hline
\end{tabular}
}
\end{center}
\end{table}

It can be seen that the Kullback-Leibler distance and the Itakura-Saito distance have computational advantage than the squared Euclidean distance for our proposed  Algorithm 3.1 in terms of number of iteration and CPU time. For one-step inertial, the computation result shows that Alg(23) ($\phi_1(x)=\varphi_2(x)$, $\phi_2(x)=\varphi_3(x)$) has the best performance while Alg(33) ($\phi_1(x)=\varphi_3(x)$, $\phi_2(x)=\varphi_3(x)$) has the poorest performance for solving the quadratic fractional programming problem. For two-step inertial, the computation result shows that Alg(23) ($\phi_1(x)=\varphi_2(x)$, $\phi_2(x)=\varphi_3(x)$) has the best performance while 
Alg(33) ($\phi_1(x)=\varphi_3(x)$, $\phi_2(x)=\varphi_3(x)$) has the poorest performance for solving the quadratic fractional programming problem.

\noindent{\bf Problem 2}. We randomly select matrix $M\in \mathbb R^{m\times m} $, and vectors $a\in \mathbb R^{m}, b\in \mathbb R^{m}$. For $m=5,20,50,100$, we give their numerical results for Algorithm 3.1 with different Bregman distances in Table 4.

\begin{sidewaystable}[htbp]\caption{Numerical results of Algorithm 3.1 for random matrix $M$ }\label{table4}
\begin{tabular}{lllllllllllllllllll}
\hline
         &  & Alg.(11) &  & Alg.(12) &  & Alg.(13) &  & Alg.(21) &  & Alg.(22) &  & Alg.(23) &  & Alg.(31) &  & Alg.(32) &  & Alg.(33) \\ \hline
$m=5$    &  &          &  &          &  &          &  &          &  &          &  &          &  &          &  &          &  &          \\
Iter.    &  & 431   &  & 432   &  & 430   &  &106    &  &111    &  &105    &  &2069    &  &1760    &  &2625    \\
InIterx. &  & 812   &  & 812   &  & 811   &  &406    &  &410    &  &405    &  &2348    &  &2078    &  &2851     \\
InItery. &  & 469   &  & 471   &  & 467   &  &298    &  &301    &  &297    &  &2286    &  &1927    &  &2935     \\
Time(s)  &  & 0.3125   &  & 0.2438   &  & 0.1813   &  & 0.1250   &  & 0.1156   &  & 0.0575   &  & 0.9844   &  & 0.7969   &  & 1.0156  \\ \hline
$m=20$   &  &          &  &          &  &          &  &          &  &          &  &          &  &          &  &          &  &          \\
Iter.    &  &445    &  &445    &  &445    &  &158    &  &156    &  &155    &  &2921    &  &2437    &  &3654    \\
InIterx. &  &872    &  &872    &  &872    &  &427    &  &429    &  &410    &  &3223    &  &2770    &  &3916    \\
InItery. &  &471    &  &470    &  &470    &  &309    &  &308    &  &307    &  &3115    &  &2588    &  &3924    \\
Time(s)  &  & 0.9375   &  & 0.2813   &  & 0.1938   &  &0.1388    &  &0.1406    &  &0.0781    &  &1.1875   &  &0.9531    &  &1.2188     \\\hline
$m=50$   &  &          &  &          &  &          &  &          &  &          &  &          &  &          &  &          &  &          \\
Iter.    &  &468    &  &468    &  &468    &  &315    &  &312    &  &310    &  &3488    &  &2812    &  &4400    \\
InIterx. &  &907    &  &907    &  &907    &  &483    &  &478    &  &475    &  &3788    &  &3149    &  &4656    \\
InItery. &  &485    &  &485    &  &485    &  &313    &  &311    &  &310    &  &3704    &  &2978    &  &4695    \\
Time(s)  &  &1.1250    &  &0.3750    &  &0.6250    &  &0.1563    &  &0.1750    &  &0.1094    &  &1.9063    &  &1.8906    &  &2.9219     \\\hline
$m=100$  &  &          &  &          &  &          &  &          &  &          &  &          &  &          &  &          &  &          \\
Iter.    &  & 651   &  & 638   &  & 657   &  &382  &  &379   &  &377  &  & 4692   &  & 3719   &  & 5529   \\
InIterx. &  & 960   &  & 947   &  & 966   &  &649   &  &643   &  &641   &  & 4903   &  & 3976   &  & 5712   \\
InItery. &  & 694   &  & 679   &  & 700   &  &385   &  &380   &  &377   &  & 4944   &  & 3911   &  & 5842   \\
Time(s)  &  &2.8750    &  &1.3906    &  &2.0313    &  &1.4375    &  &0.7656    &  &0.5938   &  &3.6094    &  &2.2500    &  &4.6875     \\\hline
\end{tabular}
\end{sidewaystable}
From Table 4, it can be seen from that the numerical results show that Alg(23) ($\phi_1(x)=\varphi_2(x)$, $\phi_2(x)=\varphi_3(x)$) performs best when the dimension of the matrix $M$ is low, and Alg(12) ($\phi_1(x)=\varphi_1(x)$, $\phi_2(x)=\varphi_2(x)$) performs a little better when the dimension of the matrix $M$ is slightly higher.
\section{Conclusion}
{Based on the proximal alternating linearized minimization algorithm, two-step inertial Bregman proximal alternating linearized minimization algorithm is proposed to solve nonconvex and nonsmooth nonseparable optimization problems. We construct appropriate benefit function and obtain that the generated sequence is globally convergent to a critical point, under the assumptions that the objective function satisfies the Kurdyka--{\L}ojasiewicz inequality and the parameters satisfy certain conditions. In numerical experiments, we choose appropriate Bregman distance such that the solution of subproblems have closed form for solving the sparse signal recovery problem. We also apply different Bregman distance to solve quadratic fractional programming problem. Numerical results are reported to show the effectiveness of the proposed algorithm.}


\end{document}